\documentclass{amsart}
\usepackage{bbm}
\usepackage{cite}

 \usepackage{hyperref}
 \usepackage{amssymb}
 \usepackage{amsmath,amsthm}
 \usepackage{latexsym}
 \usepackage{amsfonts}
 \usepackage{graphicx}
 \usepackage{mathrsfs}
 
\newtheorem{theorem}{Theorem}[section]
\newtheorem{lemma}[theorem]{Lemma}
\newtheorem{corollary}[theorem]{Corollary}
\newtheorem{proposition}[theorem]{Proposition}

\theoremstyle{definition}
\newtheorem{definition}[theorem]{Definition}
\newtheorem{example}[theorem]{Example}

\theoremstyle{remark}
\newtheorem{remark}[theorem]{Remark}

\renewcommand{\a}{\alpha}
\renewcommand{\b}{\beta}

\newcommand{\g}{\lambda}

\newcommand{\q}{\quad}

\newcommand{\M}{{\cal M}}

\newcommand{\ty}{\infty}
\newcommand{\e}{\varepsilon}

\newcommand{\ov}[1]{\overline{#1}}

\renewcommand{\O}{\Omega}

\newcommand{\eR}{\mathbb{R}}
\newcommand{\eN}{\mathbb{N}}
\newcommand{\Ze}{\mathbb{Z}}
\newcommand{\Qu}{\mathbb{Q}}
\newcommand{\Ce}{\mathbb{C}}

\newcommand{\re}{\mathop{\mathrm{Re}}}
\newcommand{\im}{\mathop{\mathrm{Im}}}

\newcommand{\po}{{\mathop{\mathcal P}}}

\newcommand{\res}{\operatorname{res}}

\newcommand{\I}{\mathbbm{i}}
\newcommand{\E}{\mathrm{e}}

\newcommand{\ovb}[1]{\mkern 1.5mu\overline{\mkern-1.5mu#1\mkern-1.5mu}\mkern 1.5mu}
\newcommand{\unb}[1]{\mkern 1.5mu\underline{\mkern-1.5mu#1\mkern-1.5mu}\mkern 1.5mu}

\newcommand{\di}{\,\mathrm{d}}

\newcommand{\cal}{\mathcal}

\begin{document}


\title[Quasiperiodic sets at infinity]{Quasiperiodic sets at infinity and meromorphic extensions of their fractal zeta functions}


\author[Goran Radunovi\'c]{Goran Radunovi\'c}
\thanks{The research of G.\ Radunovi\' c was supported by the Croatian Science Foundation grants PZS-2019-02-3055, UIP-2017-05-1020 and the bilateral Hubert-Curien Cogito grant 2021-22.}
\address{University of Zagreb, Faculty of Science, Horvatovac 102a, 10000 Zagreb, Croatia}
\email{{\tt goran.radunovic@math.hr} (Goran Radunovi\'c)}

%
%
%

\begin{abstract}
In this paper we introduce an interesting family of relative fractal drums (RFDs in short) at infinity and study their complex dimensions which are defined as the poles of their associated Lapidus (distance) fractal zeta functions introduced in a previous work by the author.

We define the tube zeta function at infinity and obtain a functional equation connecting it to the distance zeta function at infinity much as in the classical setting.
Furthermore, under suitable assumptions, we provide general results about existence of meromorphic extensions of fractal zeta functions at infinity in the Minkowski measurable and nonmeasurable case.
We also provide a sufficiency condition for Minkowski measurability as well as an upper bound for the upper Minkowski content, both in terms of the complex dimensions of the associated RFD.

We show that complex dimensions of quasiperiodic sets at infinity posses a quasiperiodic structure which can be either algebraic or transcedental.
Furthermore, we provide an example of a maximally hyperfractal set at infinity with prescribed Minkowski dimension, i.e., a set such that the abscissa of convergence of the corresponding fractal zeta function is in fact its natural boundary.
\end{abstract}

\keywords{distance zeta function, relative fractal drum, complex dimensions, Minkowski content, Minkowski dimension,}

\subjclass{11M41, 28A12, 28A75, 28A80, 28B15, 42B20, 44A05, 30D30}

\maketitle

\tableofcontents


\section{Introduction}

In this paper we study relative fractal drums $(A,\O)$ where $A:=\{\infty\}$ is the point at infinity via their Minkowski dimension, as well as their complex dimensions which are a far-reaching generalization of the Minkowski dimension.
A (classical) relative fractal drum $(A,\O)$ is an ordered pair of subsets of the $N$-dimensional Euclidean space $\mathbb{R}^N$, where $A$ is nonempty and $\Omega$ is Lebesgue measurable of finite $N$-dimensional volume (satisfying another mild technical condition).
These objects are a convenient generalization of the notion of a compact subset of $\mathbb{R}^N$ and were studied extensively in \cite{LapRaZu4,LapRaZu6,LapRaZu7} and the research monograph \cite{fzf}, along with their associated Minkowski dimension, Minkowski content and their fractal zeta functions and associated complex dimensions.
The complex dimension themselves are defined as poles (or more general singularities; see \cite{essen}) of the associated Lapidus (also called {\em distance}) zeta function of $(A,\O)$ given as the following Lebesgue integral
\begin{equation}\label{zeta_dist}
\zeta_{A,\O}(s):=\int_{\O}d(x,A)^{s-N}\di x,
\end{equation}
initially, for all $s\in\Ce$ such that $\re s$ is sufficiently large, where $d(x,A)$ is the Euclidean distance from $x$ to $A$.
The basic properties of the above integral is that it is absolutely convergent in the open half-plane $\{\re s>\ov{\dim}_B(A,\O)\}$, where $\ov{\dim}_B(A,\O)$ denotes the upper Minkowski dimension of the RFD $(A,\Omega)$ and hence, defines a holomorphic function in that half-plane.

Fore the general higher-dimensional theory of complex dimensions and fractal zeta functions we refer the reader to \cite{fzf,LapRaZu2,LapRaZu3} along with the survey articles \cite{brezish,tabarz}, as well as the relevant references therein.
This higher-dimensional theory of complex dimensions is a far-reaching generalization of the the well known theory of geometric zeta functions for fractal strings and their complex dimensions due to Michel L.\ Lapidus and his numerous collaborators (see \cite{lapidusfrank12} and the relevant references therein).

As already pointed out, we study here a degenerated type of relative fractal drums $(\infty,\O)$ where the set $A$ becomes the point at infinity.
The study of such RFDs was started in \cite{ra,ra2} where the basic notions of the associated Minkowski content and dimension at infinity was introduced, along with the corresponding notion of fractal zeta functions and complex dimensions at infinity.
Basic results about these objects were given along with a number of interesting examples.

Here we will provide further interesting results about the fractal zeta function at infinity, concretely, results about existence of its meromorphic extension beyond the initial half-plane of analyticity as well as about the connection to the notion of Minkowski content and measurability of $(\infty,\Omega)$.

The motivation is to construct an interesting family of transcendentally and algebraically quasiperiodic (in the sense of \cite{fzf}) sets at infinity, which will then be used to construct an example of a maximally hyper-fractal set at infinity in the sense that its zeta function has a natural barrier and thus, cannot be extended beyond the initial half-plane of convergence.

Intuitively, it is clear that the ``fractality'' of $(\infty,\O)$ stems solely from the set $\Omega$ as opposed to in the classical setting where the ``fractality'' of the RFD $(A,\Omega)$ stems from the set $A$ and $\Omega$ is usually chosen to be metrically associated to the set $A$ in the sense of \cite{Wi19,Sta}.
Of course, the set $A$ is usually the set one wishes to investigate, while the set $\Omega$ is usually used for localization purposes, i.e., when one wants, for instance, to analyze the part of the set $A$ ``seen'' only from the set $\Omega$.
For example, one might understand $A$ as a boundary of a fractal membrane $\Omega$, and wishes to investigate the vibrations of $\Omega$ from the point of view of spectral theory.

Nevertheless, here, as a counterpoint to the classical setting, the aforementioned examples of quasiperiodic sets at infinity show that there exist interesting families of nontrivial RFDs $(\infty,\Omega)$ from the fractal point of view even though the set $A$ is just one point at infinity.
We also note that from the point of view of the complex dimensions (i.e., fractal zeta functions), it is equivalent, in fact, to study the classical RFD $(\mathbf{O},\Phi(\Omega))$ where $\Phi\colon\mathbb{R}^N\to\mathbb{R}^N$  is the standard geometric inversion, i.e., $\Phi(x)=x/|x|^2$, and $\mathbf{O}:=\{0\}$ is the singleton containing the origin of $\mathbb{R}^N$.

In general, the motivation to study unbounded domains that have fractal properties may be found in problems from oscillation theory \cite{Dzu,Karp}, automotive industry \cite{She}, aerodynamics \cite{Dol}, civil engineering \cite{Pou} and mathematical applications in biology \cite{May}.
Also, unbounded domains are of interest in problems of partial differential equations, for instance, solvability of Dirichlet problems for quasilinear equations in unbounded domains \cite{Maz1} and \cite[Section~15.8.1]{mazja}.
See also~\cite{An,Hur,La,Rab} and~\cite{VoGoLat}.
Furthermore, fractal properties of unbounded trajectories of some planar vector fields were studied in \cite{razuzup} in connection to the Hopf bifurcation at infinity.
In that paper, the geometric inversion was used to bring the system near zero instead of near infinity and the classical Minkowski dimension of the geometrically inverted system was then studied.
Here we study different objects and work directly at infinity but show, nevertheless, that the connection with geometric inversion is also present.

The Paper is organized as follows.
In Section \ref{inf_box_def} we recall the most important definitions and results about Minkowski dimension and content of sets at infinity and their associated fractal zeta functions.
Most of these results were proved in \cite{ra2}.

In Section \ref{mero_ext_sec} we provide general results about meromorphic extension of fractal zeta functions at infinity.
We introduce the tube zeta function at infinity and provide a functional equation connecting it to the distance zeta function at infinity in  Theorem \ref{tubedistinf}.
Next, under suitable hypothesis, Theorem \ref{restubinf} establishes a connection between the residue of the tube zeta function at the point equal to the Minkowski dimension of the associated RFD and its (upper and lower) Minkowski content.
Furthermore, Theorems \ref{Mink_measurable_inf} and \ref{Mink_nonmeasurable_inf} establish, again under suitable hypotheses, the existence of a meromorphic extension of the tube zeta function at infinity under the assumption that the underlying set is Minkowski measurable and nonmeasurable, respectively.
Moreover, Theorem \ref{mink_suff} establishes a sufficient condition for the set to be Minkowski measurable at infinity while Theorem \ref{mink_bound_inf} provides an upper bound for the upper Minkowski content in terms of the residue of its tube (or distance) zeta function.
Finally Theorem \ref{inf_mink_con} establishes a connection between the Minkowski measurability of the set at infinity and its geometrically inverted image.  

Section \ref{prop_sec_ty} is dedicated to establishing some of the more technical but very useful general properties of fractal zeta functions at infinity needed later on, such as the scaling property (Proposition \ref{scaling_prop}) and general behavior under the change of the norm on $\mathbb{R}^N$ in Theorem \ref{equiv_mero}.

In the final Section \ref{qp_sets} we construct maximally hyperfractal  sets at infinity of prescribed Minkowski dimension in Theorem \ref{qp_construction}.
Then we show that these sets can be used to generate algebraically and transcedentally quasiperiodic sets at infinity of any order in Theorems \ref{n-quasi} or infinite order in Theorem \ref{ty_quasi}.
The construction of classical compace sets and RFDs which are algebraically quasiperiodic is an open problem \cite[Problem 6.2.3]{fzf}.
Here we solve it in the setting of sets at infinity and also explain how they can be used to obtain clasical RFDs which are algebraically quasiperiodic by using geometric inversion.
The problem of finding algebraically quasiperiodic compact sets is still open.

\section{Minkowski dimension and fractal zeta functions of sets at infinity}\label{inf_box_def}

In this section we recall the most important definitions and results from \cite{ra2}.
Let $\O$ be a Lebesgue measurable subset $\eR^N$ of finite Lebesgue measure, i.e., $|\O|<\infty$; $|\cdot|$ denotes the $N$-dimensional Lebesgue measure.
We let
\begin{equation}\label{kratica}
_t\O:=B_t(0)^c\cap\O,
\end{equation}
where $t>0$ and $B_t(0)^c$ is the complement of the open ball of radius $t$ centered at $0$.
For any real number $r$ one defines the {\em upper $r$-dimensional Minkowski content} of $\O$ {\em at infinity}
\begin{equation}\label{upMinkinf}
{\ovb{\M}}^{r}(\ty,\O):=\limsup_{t\to+\infty}\frac{|{_t\O}|}{t^{N+r}},
\end{equation}
and, analogously, by taking the lower limit in \eqref{upMinkinf} as $t\to +\infty$, the {\em lower $r$-dimensional Minkowski content} of $\O$ {\em at infinity} denoted by ${\unb{\M}}^{r}(\ty,\O)$.

It is easy to see that the above definition implies the existence of a unique $D\in\eR$ such that ${\ovb{\M}}^{r}(\ty,\O)=+\ty$ for $r<\ovb{D}$ and ${\ovb{\M}}^{r}(\ty,\O)=0$ for $r>\ovb{D}$ and analogously for the lower Minkowski content; see Figure \ref{fig:1}.
The value $\ovb{D}$ is called the {\em upper Minkowski dimension of $\O$ at infinity}, ${{\ovb{\dim}}}_B(\ty,\O)$ or the {\em upper Minkowski dimension} of $(\infty,\O)$, i.e., one has
\begin{equation}\label{updiminf}
\begin{aligned}
{{\ovb{\dim}}}_B(\ty,\O):=&\sup\{r\in\eR:{\ovb{\M}}^{r}(\ty,\O)=+\ty\}\\
=&\inf\{r\in\eR:{\ovb{\M}}^{r}(\ty,\O)=0\},
\end{aligned}
\end{equation}
and similarly for the lower analog denoted by ${\unb{\dim}}_B(\ty,\O)$.
If the upper and lower Minkowski dimensions coincide, we say that the {\em Minkowski dimension} of $(\infty,\O)$ exists and denote it by $\dim_B(\ty,\O)$.

Furthermore, in the case when the upper and lower $r$-dimesninal Minkowski contents of $(\infty,\O)$ coincide we say that the {\em $r$-dimensional Minkowski content} of $(\infty,\O)$ exists and denote it by $\M^{r}(\ty,\O)$.
Moreover, in the case when 
$
0<{\unb{\M}}^{D}(\ty,\O)\leq{\ovb{\M}}^{D}(\ty,\O)<+\ty,
$
for some $D\in\eR$ (necessarily $D=\dim_B(\ty,\O)$ in that case), we say that $(\infty,\O)$ is {\em Minkowski nondegenerate}.
Finally, $(\infty,\O)$ is said to be {\em Minkowski measurable} if it is Minkowski nondegenerate and its lower and upper Minkowski content coincide.

The next simple facts were proved in \cite{ra2} and we recall them here for completeness.
For any Lebesgue measurable $\O\subseteq\eR^N$ one has that $-\infty\leq{\unb{{\dim}}}_B(\ty,\O)\leq{{\ovb{\dim}}}_B(\ty,\O)\leq-N$.
Furthermore, both, $-\infty$ and $-N$ can be attained; see \cite[Example 3 and Proposition 2]{ra2}.
Moreover, if ${\ovb{\dim}}_B(\ty,\O)=-N$, then one always has that ${\ovb{\M}}^{-N}(\ty,\O)=0$ which follows directly from the definition.
This should not be surprising since the set $A$ is just a point at infinity, hence, one should not expect that the Minkowski dimension of $(\infty,\O)$ can be larger than $0$.
The fact that it cannot be larger than $-N$ is actually connected to the fact that $\Omega$ has finite volume.
We will show in a future paper that the ``dimensional gap interval'' $(-N,0]$ is actually ``reserved'' for RFDs $(\infty,\Omega)$ where we let $\Omega$ to have infinite volume.
Of course, in that case the definition of its Minkowski content and the corresponding fractal zeta functions must be modified accordingly since $|_t\O|$ is infinite.

We point out that also classical RFDs with negative dimension exist; see \cite{fzf} where this feature is explained by the lack of the so-called cone property. 
Although the dimension of $(\infty,\Omega)$ is always negative and therefore, seems uninteresting at first, we will show that there exist rich families of unbounded sets $\Omega$ whose complex dimensions have complicated quasiperiodic structures.
Therefore, the source of ``fractality'' in the sense of Lapidus, i.e., the fact that $(\infty,\O)$ possesses non-real complex dimensions, stems solely from the unbounded set $\Omega$.
This also shows that in general, one has to be careful since the source of ``fractality'' of an RFD $(A,\O)$ could be from both sets, $A$ and $\O$.
On the other hand, we conjecture that this cannot happen if $\O$ is metrically associated to $A$.

\begin{figure}[h]
\begin{center}
\includegraphics[width=10cm]{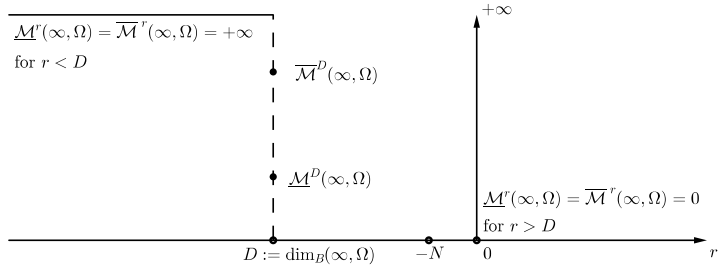}
\end{center}
\label{fig:1}
\caption{The graphs of the functions $r\mapsto\ovb{\mathcal{M}}^{r}(\ty,\O)$ and $r\mapsto\unb{\mathcal{M}}^{r}(\ty,\O)$, assuming that $\O$ is Minkowski
nondegenerate and nonmeasurable at infinity, that is, $D:=\dim_B(\ty,\O)$ exists and $0<\unb{\mathcal{M}}^{r}(\ty,\O)<\ovb{\mathcal{M}}^{r}(\ty,\O)<\ty$.}
\label{mink_besk}
\end{figure}

The next two examples from \cite[Examples 1 and 2]{ra2} give a glimpse of nontrivial RFDs at infinity.

\begin{example}
	Let $\a>0$ and $\b>1$ be fixed and let
$
a_j:=j^{\a}$, $l_j:=j^{-\b}$, $b_j:=a_j+l_j
$ and $I_j:=(a_j,b_j)$ for $j\in\mathbb{N}$.
Consider
\begin{equation}\label{omega_a_b}
\O(\a,\b):=\bigcup_{j=1}^{\ty}I_j\subseteq\eR,
\end{equation}
then,
\begin{equation}\label{dim(a,b)}
D:=\dim_B(\ty,\O(\a,\b))=\frac{1-(\a+\b)}{\a}\quad\textrm{ and }\quad\M^{D}(\ty,\O(\a,\b))=\frac{1}{\b-1}.
\end{equation}
Observe that by varying parameters $\a$ and $\b$, we can obtain any value in $(-\ty,-1)$ for $\dim_B(\ty,\O(\a,\b))$.
\end{example}

The next example will be one of the crucial building blocks for the construction of quasiperiodic sets at infinity.

\begin{example}\label{standardniprim}
For $\a>1$ let $\O:=\{(x,y)\in\eR^2\,:\,x>1,\ 0<y<x^{-\a}\}$. Then we have that
\begin{equation}\label{dimstandardniprim}
D:=\dim_B(\ty,\O)=-1-\a\quad \mathrm{ and }\quad \M^{D}(\ty,\O)=\frac{1}{\a-1}.
\end{equation}
Observe that $\dim_B(\ty,\O)\to -\ty$ and $\M^{D}(\ty,\O)\to 0$ as $\alpha\to +\ty$
\end{example}

In general, the notion of Minkowski dimension at infinity and Minkowski nondegeneracy at infinity do not depend on the choice of the norm on $\eR^N$ in which we define the ball $B_t(0)$.
More precisely, if $K_t(0)^c$ denotes a complement of a ball in another (necessarily, equivalent) norm $\|\cdot\|$ on $\eR^N$, we define the analog of the (upper and lower) Minkowski content at infinity by using $|K_t(0)^c\cap\Omega|$ in \eqref{upMinkinf} instead of $|_t\O|$, the two Minkowski dimensions of $(\infty,\Omega)$ will coincide.
Furthermore, the notion of Minkowski nondegeneracy is invariant by such changes of the norm; see \cite[Lemma 2 and Corollary 2]{ra2}.

In the remainder of this section we recall the definition and basic properties of the {\em Lapidus zeta function} at infinity (also called the {\em distance zeta function of $(\infty,\O)$}) from \cite[Section 3]{ra2} defined by the Lebesgue integral
\begin{equation}\label{infzeta}
\zeta_{\ty,\O}(s):=\zeta_{\ty,\O}(s;T)=\int_{{B_T(0)^c\cap\O}}|x|^{-s-N}\di x,
\end{equation}
for some fixed $T>0$ and $s$ in $\Ce$ with $\re s$ sufficiently large.

The dependence on $T>0$ is not important since changing $T$ amounts to adding an entire function to \eqref{infzeta} and we are only interested in possible singularities of \eqref{infzeta}.
The distance zeta function of $(\infty,\O)$ is closely related to the classical distance zeta function of the ``geometrically inverted'' relative fractal drum $(\mathbf{O},\Phi(\Omega))$.
More precisely, they are connected by a functional equation \cite[Theorem 3]{ra2}: $\zeta_{\ty,\O}(s;T)=\zeta_{\mathbf{O},\Phi(\O)}(s;1/T)$; hence, from the point of view of complex dimensions it is completely equivalent to either study $(\infty,\Omega)$ or its geometric inversion $(\mathbf{O},\Phi(\O))$.

We now state a part of the holomorphicity theorem \cite[Theorem 5]{ra2} for the distance zeta function at infinity for the sake of exposition.
Recall also that we define the {\em abscissa of convergence} of $\zeta_{\infty,\O}$ as the infimum of all $\sigma\in\mathbb{R}$ such that the integral \eqref{infzeta} is absolutely convergent for all $s\in\mathbb{C}$ such that $\re s>\sigma$ and we denote it by $D(\zeta_{\infty,\O})$.

\begin{theorem}[{Holomorphicity theorem \cite[Theorem 5]{ra2}}]\label{analiticinf}
Let $\O$ be any Lebesgue measurable subset of $\eR^N$ of finite $N$-dimensional Lebesgue measure.
Assume that $T$ is a fixed positive number.
Then the following conclusions hold.

\noindent $(a)$ The abscissa of convergence of the Lapidus zeta function at infinity
\begin{equation}\label{inteqzeta}
\zeta_{\ty,\O}(s)=\int_{{_T\O}}|x|^{-s-N}\di x
\end{equation}
is equal to the upper box dimension of $\O$ at infinity, i.e.,
\begin{equation}
D(\zeta_{\ty,\O})=\ovb{\dim}_B(\ty,\O).
\end{equation}
Consequently, $\zeta_{\ty,\O}$ is holomorphic on the half-plane $\{\re s>{\ovb{\dim}}_B(\ty,\O)\}$.

\noindent $(b)$ The half-plane from $(a)$ is optimal.\footnote{Optimal in the sense that the integral appearing in~\eqref{inteqzeta} is diveregent for real $s\in(-\ty,\ovb{D})$.}

\noindent $(c)$ If $D=\dim_B(\ty,\O)$ exists and ${\unb{\M}}^{D}(\ty,\O)>0$, then $\zeta_{\ty,\O}(s)\to+\ty$ for $s\in\eR$ as $s\to D^+$.
\end{theorem}

\begin{remark}
In Theorem~\ref{analiticinf} one may replace the norm appearing in the definition of the distance zeta function at infinity and in the definition od the Minkowski content at infinity by any other norm on $\eR^N$ to obtain a completely analog result.
Furthermore by \cite[Proposition 6]{ra2} the difference $\zeta_{\ty,\O}(s;|\,\cdot\,|_{\ty})-\zeta_{\ty,\O}(s)$ is holomorphic at least on the half-plane $\{\re s>\ovb{\dim}_B(\infty,\Omega)-2\}$, where $\zeta_{\ty,\O}(s;|\,\cdot\,|_{\ty})$ is the distance zeta function of $(\infty,\O)$ defined by using the $\infty$-norm on $\mathbb{R}^N$.
This result is very practical for obtaining results about the (Euclidean) distance zeta function $\zeta_{\ty,\O}(s)$ and determining its poles in that half-plane.
Since a complete proof of this result was not given in \cite{ra2} due to space constrains we give the proof of a more general results; Theorem \ref{equiv_mero} and Proposiion \ref{euc_ty} from which \cite[Proposition 6]{ra2} follows directly.
\end{remark}

For example, the distance zeta function of $(\infty,\Omega)$ from Example \ref{standardniprim} can be computed explicitly (in the $|\cdot|_{\infty}$-norm on $\eR^2$) and is shown to be meromorphic in all of $\Ce$ and given by $\zeta_{\infty,\Omega}(s)=1/(s+\alpha+1)$, having a single simple pole at $s=-1-\alpha$.

Next we recall the notion of complex dimensions of $(\ty,\O)$ from \cite{ra2}.
Namely, if the distance zeta function $\zeta_{\infty,\O}$ possesses a meromorphic extension to some open connected neighborhood $W$ (usually called the {\em window}) of the half-plane $\{\re s\geq\ovb{\dim}_B(\ty,\O)\}$, one defines the {\em set of visible complex dimensions of $(\ty,\O)$ through $W$} as the set of poles of the distance zeta function $\zeta_{\ty,\O}$ that are contained in $W$ and denote it by
\begin{equation}\label{po_vis}
\po(\zeta_{\ty,\O},W):=\{\omega\in W:\mbox{$\omega$ is a pole of }\zeta_{\ty,\O}\}
\end{equation}
which one usually abbreviates to $\po(\zeta_{\ty,\O})$ if there is no ambiguity concerning the choice of $W$ (or when $W=\Ce$).

The subset of $\po(\zeta_{\ty,\O},W)$ consisting of poles with real part equal to $\ovb{\dim}_B(\ty,\O)$ is then called the set of {\em principal complex dimensions of $(\ty,\O)$} and denoted by $\dim_{PC}(\ty,\O)$.

\section[Meromorphic extensions of zeta functions and complex dimensions]{Meromorphic extensions of fractal zeta functions at infinity}\label{mero_ext_sec}

In general it is difficult to obtain a closed formula for the distance zeta function which would then induce its mermorphic extension to a larger domain containing the original half-plane of holomorphicity.

In this section we will give sufficient conditions on the Lebesgue measurable set $\O\subseteq\eR^N$ of finite Lebesgue measure which will ensure that the Lapidus zeta function of $\O$ at infinity has a meromorphic continuation to a neighborhood of its critical line.
Firstly, we will state and prove the theorems in terms of the tube zeta function at infinity and then, by using the functional equation between the Lapidus (distance) and the tube zeta function at infinity (see Theorem~\ref{tubedistinf}), we will obtain the analog statements in terms of the Lapidus zeta function at infinity.

Furthermore, we will also give a sufficient condition for a relative fractal drum $(\ty,\O)$ to be Minkowski measurable at infinity in terms of its distance or tube zeta function at infinity.

In \cite[Theorem 6]{ra2} it was already shown, under mild hypotheses, that the residue of the distance zeta function of $(\infty,\O)$ is closely related to its (upper and lower) Minkowski content at infinity.
In particular, under suitable hypotheses we have that $\res(\zeta_{\ty,\O},D)=-(N+D)\M^{D}(\ty,\O)$ where $D=\dim_B(\infty,\O)$.

Let $\O$ be a Lebesgue measurable subset of $\eR^N$ and $|\O|<\ty$.
Similarly as in the case of standard relative fractal drums \cite{fzf} we define the {\em tube zeta function} of $\O$ {\em at infinity} and denote it with $\widetilde{\zeta}_{\ty,\O}$:
\begin{equation}\label{tubeinf}
\widetilde{\zeta}_{\ty,\O}(s;T):=\int_{T}^{+\ty}t^{-s-N-1}|_t\O|\di t,
\end{equation}
where $T>0$ is fixed.
The next theorem establishes the aforementioned functional equation, from which the analyticity of the tube zeta function will follow.

\begin{theorem}[Functional equation between tube and distance zeta functions at infinity]\label{tubedistinf}
Let $\O\subseteq\eR^N$ with $|\O|<\ty$ and let $T>0$ be fixed.
Then for every $s\in\Ce$ such that $\re s>{\ovb{\dim}}_B(\ty,\O)$ it holds that
\begin{equation}\label{tubedistinfeq}
\int_{_T\O}|x|^{-s-N}\di x=T^{-s-N}|_T\O|-(s+N)\int_T^{+\ty}t^{-s-N-1}|_t\O|\di t,
\end{equation}
i.e., the following functional equation holds:
\begin{equation}
\zeta_{\ty,\O}(s;T)=T^{-s-N}|_T\O|-(s+N)\widetilde{\zeta}_{\ty,\O}(s;T).
\end{equation}
\end{theorem}

\begin{proof}
Firstly, from \cite[Proposition 4]{ra2} we have that~\eqref{tubedistinfeq} is valid for a real number $s$ such that $s>{\ovb{\dim}}_B(\ty,\O)$.
To show that the equality holds in the half-plane $\{\re s>{\ovb{\dim}}_B(\ty,\O)\}$ it suffices to prove that both sides of Equation~\eqref{tubedistinfeq} are holomorphic functions on that domain.
The left-hand side of~\eqref{tubedistinfeq} is holomorphic in $\{\re s>{\ovb{\dim}}_B(\ty,\O)\}$ according to Theorem~\ref{analiticinf}.
The same is valid for the right-hand side of~\eqref{tubedistinfeq}.
Namely, this is a Dirichlet type integral with $\varphi(t)=t^{-s}$ and $\di\mu(t)=t^{-N-1}|_t\O|\di t$, and according to \cite[Theorem 2.1.44$(a)$]{fzf} it is sufficient to show that this integral is absolutely convergent for $\re s>{\ovb{\dim}}_B(\ty,\O)$.

For $\ovb{D}:={\ovb{\dim}}_B(\ty,\O)$ and $s\in\Ce$ such that $\re s>\ovb{D}$, let us choose $\e>0$ sufficiently small such that $\re s>\ovb{D}+\e$.
Since ${\ovb{\M}}_{\ty}^{\ovb{D}+\e}(\O)=0$, there exists a constant $C_T>0$ such that $|_t\O|\leq C_Tt^{N+\ovb{D}+\e}$ for every $t\in(T,+\ty)$. 
Now we have the following estimate:
\begin{equation}
\begin{aligned}
|\widetilde{\zeta}_{\ty,\O}(s;T)|&\leq\int_{T}^{+\ty}t^{-\re s-N-1}|_t\O|\di t\leq C_T\int_{T}^{+\ty}t^{-\re s-N-1}t^{N+\ovb{D}+\e}\di t\\
&=C_T\int_{T}^{+\ty}t^{\ovb{D}+\e-\re s-1}\di t=C_T\frac{T^{\ovb{D}+\e-\re s}}{\re s-(\ovb{D}+\e)}<+\ty.
\end{aligned}
\end{equation}
This completes the proof of the theorem.
\end{proof}

\begin{remark}
In light of the functional equation between the distance and tube zeta functions of $\O$ at infinity stated in Theorem~\ref{tubedistinf}, it is clear that the definition of (principal) complex dimensions can be also stated in terms of the tube zeta function $\widetilde{\zeta}_{\ty,\O}$ at infinity instead in terms of the distance zeta function $\zeta_{\ty,\O}$ at infinity.
Moreover, we have that
\begin{equation}
\dim_{PC}(\ty,\O)=\po_c(\zeta_{\ty,\O}=\po_c(\widetilde{\zeta}_{\ty,\O})
\end{equation}
and
\begin{equation}
\po(\zeta_{\ty,\O},W)=\po(\widetilde{\zeta}_{\ty,\O},W).
\end{equation}
\end{remark}

The next theorem is a consequence and analog of \cite[Theorem 6]{ra2} in terms of the tube zeta function of $(\infty,\O)$, its residue at $s=\dim_B(\ty,\O)$ and its upper and lower Minkowski contents.

\begin{theorem}[Residue and Minkowski content connection]\label{restubinf}
Let $\O\subseteq\eR^N$ be such that $|\O|<\ty$, $\dim_B(\ty,\O)=D<-N$ and $0<{\unb{\M}}^{D}(\ty,\O)\leq{\ovb{\M}}^{D}(\ty,\O)<\ty$.
If $\widetilde{\zeta}_{\ty,\O}$ has a meromorphic continuation to a neighborhood of $s=D$, then $D$ is a simple pole and it holds that
\begin{equation}\label{mink_res_inf_tube}
{\unb{\M}}^{D}(\ty,\O)\leq\res(\widetilde{\zeta}_{\ty,\O},D)\leq{\ovb{\M}}^{D}(\ty,\O).
\end{equation}
Moreover, if $\O$ is Minkowski measurable at infinity, then we have
\begin{equation}
\res(\widetilde{\zeta}_{\ty,\O},D)=\M^{D}(\ty,\O).
\end{equation}
\end{theorem}

\begin{proof}
Using the fact that $\zeta_{\ty,\O}(s)=T^{-s-N}|_T\O|-(s+N)\widetilde{\zeta}_{\ty,\O}(s)$ for every $s\in\Ce$ such that $\re s>D$ (proved in Theorem~\ref{tubedistinf}) and by using \cite[Theorem 6]{ra2}, we immediately have
$$
\res(\zeta_{\ty,\O},D)=\lim_{s\to D}(s-D)\left[T^{-s-N}|_T\O|-(s+N)\widetilde{\zeta}_{\ty,\O}(s)\right],
$$
i.e.,
$$
\res(\zeta_{\ty,\O},D)=-(N+D)\res(\widetilde{\zeta}_{\ty,\O},D).
$$
\end{proof}

The next theorem gives a sufficiency condition for the existence of a meromorphic continuation of the tube zeta function at infinity in terms of the asymptotics of the tube function $t\mapsto|_t\O|$ as $t\to +\infty$.

\begin{theorem}[Meromorphic extension - Minkowski measurable case]\label{Mink_measurable_inf}
Let $\O\subseteq\eR^N$ be a Lebesgue measurable set of finite Lebesgue measure such that there exist $\a>0$, $\M\in(0,+\ty)$ and $D<-N$ satisfying
\begin{equation}
|_t\O|=t^{N+D}(\M+O(t^{-\a}))\quad\textrm{ as }t\to +\ty.
\end{equation}
Then, $\dim_B(\ty,\O)$ exists and $\dim_B(\ty,\O)=D$.
Furthermore, $\O$ is Minkowski measurable at infinity with Minkowski content $\M^{D}(\ty,\O)=\M$.
Moreover, the tube zeta function $\widetilde{\zeta}_{\ty,\O}$ has for abscissa of convergence $D(\widetilde{\zeta}_{\ty,\O})=\dim_B(\ty,\O)=D$ and possesses a unique meromorphic continuation $($still denoted by $\widetilde{\zeta}_{\ty,\O})$ to $($at least$)$ the open half-plane $\{\re s>D-\a\}$.
The only pole of $\widetilde{\zeta}_{\ty,\O}$ in this half-plane is $s=D;$ it is simple, and
$$
\res(\widetilde{\zeta}_{\ty,\O},D)=\M.
$$
\end{theorem}

\begin{proof}
For a fixed $T>0$ we have
$$
\begin{aligned}
\widetilde{\zeta}_{\ty,\O}(s)&=\int_T^{+\ty}t^{-s-N-1}|_t\O|\di t=\int_T^{+\ty}t^{-s-N-1}t^{N+D}(\M+O(t^{-\a}))\di t\\
&=\M\int_{T}^{+\ty}t^{D-s-1}\di t+\int_{T}^{+\ty}t^{-s}O(t^{D-\a-1})\di t\\
&=\underbrace{\frac{\M T^{D-s}}{s-D}}_{\zeta_1(s)}+\underbrace{\int_{T}^{+\ty}t^{-s}O(t^{D-\a-1})\di t}_{{\zeta_2(s)}}
\end{aligned}
$$
provided $\re s>D$.
The function $\zeta_1$ is meromorphic in the entire complex plane and $D(\zeta_1)=D$.
Furthermore, for the function $\zeta_2$ we have that
$$
|\zeta_2(s)|\leq K\int_T^{+\ty}t^{D-\re s-\a-1}\di t<\ty
$$
for $\re s>D-\a$ and $K$ a positive constant.
Therefore, $D(\zeta_2)\leq D-\a<D=D(\zeta_1)$ and the claim of the theorem now follows immediately.
\end{proof}

\begin{remark}
We point out that, in general, if there exists a holomorphic extension of $\widetilde{\zeta}_{\ty,\O}$ to an open domain $U\subseteq\Ce$, then we may assume that $U$ is symmetric with respect to the real axis and any isolated singularities of $\widetilde{\zeta}_{\ty,\O}$ in $\ov{U}$ (the closure of $U$) come in complex conjugate pairs.
Namely, it is clear that for real $s$, the values of $\widetilde{\zeta}_{\ty,\O}(s)$ are also real and hence, by the {\em the principle of reflection} (see, e.g., \cite[p.\ 155]{titch}), we deduce that for
all complex numbers $s$ such that $\re s>\ovb{\dim}_B(\ty,\O)$, we have $\ov{\widetilde{\zeta}_{\ty,\O}(s)}=\widetilde{\zeta}_{\ty,\O}(\ov{s})$.
Naturally, this identity remains valid upon holomorphic continuation (in whichever domain $U\subseteq\Ce$ the tube zeta function $\widetilde{\zeta}_{\ty,\O}$ can be holomorphically extended).
The analog is also true for the distance zeta function at infinity and, in general, for the other fractal zeta functions considered in the literature. 
\end{remark}

One would like to show that Minkowski measurability of $\O$ at infinity can be characterized by its tube (or distance) zeta function at infinity similarly as it was done in \cite[Chapter 5]{fzf} for classical relative fractal drums.
One direction of this result is consequence of the Wiener--Pitt Tauberian theorem.
The other direction is partially resolved by Theorem~\ref{Mink_measurable_inf}, where one has to add the additional assumption on the asymptotics of the tube formula of the $(\infty,\O)$.
For the general case one would need to express the tube formula of $(\infty,\O)$ in terms of its tube or distance zeta function.
We leave this for future work where we expect that the technique of inverse Mellin transform from \cite[Chapter 5]{fzf} will give the desired result.


Next we state and prove the announced sufficiency condition for Minkowski measurability at infinity.

\begin{theorem}[Sufficient condition for Minkowski measurability at infinity]\label{mink_suff}
Let $\O$ be a subset of $\eR^N$ of finite Lebesgue measure such that and let $\ovb{\dim}_B(\ty,\O)=\ov{D}<-N$.
Furthermore, assume that the relative tube zeta function $\widetilde{\zeta}_{\ty,\O}$ of $(\ty,\O)$ can be meromorphically extended to a neighborhood $U$ of the critical line $\{\re s=\ov{D}\}$.
Let $\ov{D}$ be its only pole in $U$ and assume that it is simple.
Then $D:=\dim_B(\ty,\O)$ exists, $D=\ov{D}$ and $(\ty,\O)$ is Minkowski measurable with
\begin{equation}
\M^{D}(\ty,\O)=\res(\widetilde{\zeta}_{\ty,\O},D).
\end{equation}

Moreover, the theorem is also valid if we replace the relative tube zeta function with the relative distance zeta function ${\zeta}_{\ty,\O}$ of $(\ty,\O)$ and in that case we have
\begin{equation}
\M^{D}(\ty,\O)=\frac{\res({\zeta,\O}_{\ty},D)}{-(N+D)}.
\end{equation}
\end{theorem}

\begin{proof}
We start with the tube zeta function $\widetilde{\zeta}_{\ty,\O}(\,\cdot\,;T)$  (choosing $T=1$ without loss of generality) and change the variable of integration by $v=\log t$.
\begin{equation}
\begin{aligned}
\widetilde{\zeta}_{\ty,\O}(s+\ovb{D})&=\int_1^{+\ty}t^{-s-D-1-N}|_t\O|\di t\\
&=\int_0^{+\ty}\E^{-sv}\E^{-v(\ovb{D}+N)}|_{\E^v}\O|\di v\\
&=\{\mathfrak{L}\sigma\}(s).
\end{aligned}
\end{equation}
where $\sigma(v):=\E^{-v(\ovb{D}+N)}|_{\E^v}\O|$.
Furthermore, from the definition of the tube zeta function of $\O$ at infinity it is clear that its residue at $s=\ovb{D}$ is real and positive.
Since $s=\ovb{D}$ is the only pole of $\widetilde{\zeta}_{\ty,\O}$ in $U$, we conclude that
\begin{equation}
G(s):=\widetilde{\zeta}_{\ty,\O}(s+\ovb{D})-\frac{\res(\widetilde{\zeta}_{\ty,\O},\ovb{D})}{s}
\end{equation}
is holomorphic on the neighborhood $\widetilde{U}:=U-\{\ovb{D}\}$ of the critical line $\{\re s\geq 0\}$ so that we can apply the Wienner-Pitt Tauberian theorem \cite[Chapter III, Lemma 9.1 and Proposition 4.3]{Kor} in its stronger form; that is, for arbitrary large $\lambda>0$ (in the notation of \cite[Chapter III, Lemma 9.1]{Kor}) and conclude that\footnote{Alternatively, see \cite[Theorem 5.4.1]{fzf} where the Wienner-Pitt Theorem was cited in a more compact form.}
\begin{equation}
\sigma_h(u)=\frac{1}{h}\int_u^{u+h}\sigma(v)\di v\to\res(\widetilde{\zeta}_{\ty,\O},\ovb{D})\quad\mathrm{as}\quad u\to +\ty,
\end{equation}
for every $h>0$.
In particular, since $v\mapsto|_{\E^v}\O|$ is nonincreasing, we have that
\begin{equation}
\begin{aligned}
\frac{1}{h}\int_u^{u+h}\E^{-v(\ovb{D}+N)}|_{\E^v}\O|\di v&\leq\frac{|_{\E^u}\O|}{h}\int_u^{u+h}\E^{-v(\ovb{D}+N)}\di v\\
&=\frac{|_{\E^u}\O|}{\E^{u(\ovb{D}+N)}}\,\frac{\E^{-h(\ovb{D}+N)}-1}{-(\ovb{D}+N)h}
\end{aligned}
\end{equation}
and by taking the lower limit of both sides as $u\to +\ty$ we get
\begin{equation}
\res(\widetilde{\zeta}_{\ty,\O},\ovb{D})\leq\unb{\mathcal{M}}^{\ovb{D}}(\ty,\O)\frac{\E^{-h(\ovb{D}+N)}-1}{-(\ovb{D}+N)h}.
\end{equation}
Since this is true for every $h>0$, by letting $h\to 0^+$ we get that
\begin{equation}\label{liminf_nej_ty}
\res(\widetilde{\zeta}_{\ty,\O},\ovb{D})\leq\unb{\mathcal{M}}^{\ovb{D}}(\ty,\O).
\end{equation}
On the other hand, we have
\begin{equation}\label{ocjena_integrala_inf}
\begin{aligned}
\frac{1}{h}\int_u^{u+h}\E^{-v(\ovb{D}+N)}|_{\E^v}\O|\di v&\geq\frac{|_{\E^{u+h}}\O|}{h}\int_u^{u+h}\E^{-v(\ovb{D}+N)}\di v\\
&=\frac{|_{\E^{u+h}}\O|}{\E^{(u+h)(\ovb{D}+N)}}\,\frac{1-\E^{h(\ovb{D}+N)}}{-(\ovb{D}+N)h}
\end{aligned}
\end{equation}
and, similarly as before, by taking the upper limit of both sides as $u\to +\ty$ we get
\begin{equation}
\res(\widetilde{\zeta}_{\ty,\O},\ovb{D})\geq\ovb{\mathcal{M}}^{\ovb{D}}(\ty,\O)\frac{1-\E^{h(\ovb{D}+N)}}{-(\ovb{D}+N)h}.
\end{equation}
Finally, since this is true for every $h>0$, we let $h\to 0^+$ and conclude that
\begin{equation}
\res(\widetilde{\zeta}_{\ty,\O},\ovb{D})\geq\ovb{\mathcal{M}}^{\ovb{D}}(\ty,\O).
\end{equation}
From this, together with~\eqref{liminf_nej_ty}, we have that $\O$ is $\ovb{D}$-Minkowski measurable at infinity and, a fortiori, that $\dim_B(\ty,\O)=D=\ovb{D}$.
Furthermore, $\res(\widetilde{\zeta}_{\ty,\O},D)={\mathcal{M}}^D(\ty,\O)$.
The part of the theorem dealing with the distance zeta function at infinity follows now from Theorem~\ref{tubedistinf} and the relation $\res({\zeta}_{\ty,\O},D)=-(N+D)\res(\widetilde{\zeta}_{\ty,\O},D)$.
\end{proof}

\begin{remark}\label{integral_remark_infty}
The assumptions of Theorem~\ref{mink_suff} can be weakened.
More precisely, it suffices to assume that
\begin{equation}
\widetilde{\zeta}_{\ty,\O}(s)-\frac{\res(\widetilde{\zeta}_{\ty,\O},\ovb{D})}{s-\ovb{D}}
\end{equation}
converges to a boundary function $G(\im s)$ as $\re s\to\ovb{D}^+$ such that
\begin{equation}
\int_{-\lambda}^{\lambda}|G(\tau)|\di\tau<\ty
\end{equation}
for every $\lambda>0$.
\end{remark}

The case when, besides $\ovb{D}$, there are other singularities on the critical line $\{\re s=\ovb{D}\}$ of the relative fractal drum $(\ty,\O)$, the Wienner-Pitt Tauberian theorem can be used to derive an upper bound for the upper $\ovb{D}$-Minkowski content of $(\ty,\O)$.
This is stated precisely in the next theorem.

\begin{theorem}[Bound for the upper Minkowski content at infinity in terms of complex dimensions]\label{mink_bound_inf}
Let $\O$ be a subset of $\eR^N$ of finite Lebesgue measure and let $\ov{D}:=\ovb{\dim}_B(\ty,\O)<-N$.
Furthermore, assume that the relative tube zeta function $\widetilde{\zeta}_{\ty,\O}$ of $(\ty,\O)$ can be meromorphically extended to a neighborhood $U$ of the critical line $\{\re s=\ovb{D}\}$ and that $\ovb{D}$ is its simple pole.
Assume also that $\{\re s=\ovb{D}\}$ contains another pole of $\widetilde{\zeta}_{\ty,\O}$ different from $\ovb{D}$.
Furthermore, let
\begin{equation}
\lambda_{(\ty,\O)}:=\inf\left\{|\ovb{D}-\omega|\,:\,\omega\in\dim_{PC}(\ty,\O)\setminus\left\{\ovb{D}\right\}\right\}
\end{equation}

Then, we have the following bound for the upper $\ovb{D}$-dimensional Minkowski content of $(\ty,\O):$
\begin{equation}\label{le_claim_inf}
\ovb{\mathcal{M}}^{\ovb{D}}(\ty,\O)\leq\frac{-(N+\ovb{D})C\lambda_{(\ty,\O)}}{2\pi\left(1-\E^{{2\pi(N+\ovb{D})}/{\lambda_{(\ty,\O)}}}\right)}\res(\widetilde{\zeta}_{\ty,\O},\ovb{D}),
\end{equation}
where $C$ is a positive constant such that $C<3$.
Equivalently, in one can obtain the upper bound in terms of distance zeta function of $(\ty,\O)$$:$
\begin{equation}
\ovb{\mathcal{M}}^{\ovb{D}}(\ty,\O)\leq\frac{C\lambda_{(\ty,\O)}}{2\pi\left(1-\E^{{2\pi(N+\ovb{D})}/{\lambda_{(\ty,\O)}}}\right)}\res({\zeta}_{\ty,\O},\ovb{D}).
\end{equation}
\end{theorem}

\begin{proof}
We use the same reasoning as in the proof of Theorem~\ref{mink_suff} with the only difference being in the fact that now we can only use the weaker statement of \cite[Chapter III, Lemma 9.1]{Kor} (or \cite[Theorem 5.4.1]{fzf}) since we have another pole on the critical line $\{\re s=\ov{D}\}$.
More precisely, if $\lambda<\lambda_{(\ty,\O)}$, then for every $h\geq 2\pi/\lambda$ by using~\eqref{ocjena_integrala_inf} and~\cite[Chapter III, Lemma 9.1, Eq.\ (9.1)]{Kor} (alternatively, \cite[Theorem 5.4.1, Eq.\ (5.4.3)]{fzf}) we have
\begin{equation}\label{poomm_ty}
C\res(\widetilde{\zeta}_{A,\O},\ovb{D})\geq\ovb{\mathcal{M}}^{\ovb{D}}(A,\O)\frac{1-\E^{h(N+\ovb{D})}}{-(N+\ovb{D})h},
\end{equation}
where $C$ is a positive constant such that $C<3$.
Since the right-hand side above is decreasing in $h$, we get the best estimate for $h=2\pi/\lambda$.
Moreover, since this is true for every $\lambda<\lambda_{(\ty,\O)}$, we get~\eqref{le_claim_inf} by letting $\lambda\to\lambda_{(\ty,\O)}^-$.
Finally, the statement about the relative distance zeta function follows by the same argument as in Theorem~\ref{mink_suff}.
\end{proof}

\begin{remark}
Similarly as in the case of Theorem~\ref{mink_suff} (see Remark~\ref{integral_remark_infty}) the hypotheses of Theorem~\ref{mink_bound_inf} can be considerably weakened.
We have stated it in this form since this is the most common case we encounter in our examples.
For instance, to bound the upper $\ovb{D}$-dimensional Minkowski content of $(\ty,\O)$ one may only assume that the relative tube or distance zeta function of $(\ty,\O)$ can be holomorphically continued to a pointed disk $B_r(\ovb{D})\setminus\{\ovb{D}\}$.
In that case~\eqref{le_claim_inf} is valid with $\lambda_{(\ty,\O)}$ replaced with the radius $r$.
Of course, the bigger the radius of the disc, the better the bound.
All one actually needs is the $L^1$-convergence of the tube (or distance) zeta function of $(\ty,\O)$ to a boundary function defined on a symmetric vertical interval $(\ovb{D}-r\I,\ovb{D}+r\I)$ as $\re s\to\ovb{D}^+$, similarly as in Remark~\ref{integral_remark_infty}.
\end{remark}

As was already mentioned in the introduction, there is a deep connection between the $(\infty,\O)$ and its image under the geometric inversion.
In light of \cite[Theorem 3]{ra2} and Theorem \ref{restubinf} we obtain a more precise connection between these two RFDs in the case when $(\ty,\O)$ is Minkowski measurable at infinity and its zeta function has a meromorphic continuation to a neighborhood of $D=\dim_B(\ty,\O)$.

\begin{theorem}[Geometric inversion and Minkowski measurability]\label{inf_mink_con}
Let $\O\subseteq\eR^N$ be such that $|\O|<\ty$ and $\dim_B(\ty,\O)=D<-N$ such that it is Minkowski measurable at infinity and assume that 
 $\zeta_{\ty,\O}$ has a meromorphic continuation to a neighborhood of $s=D$.
Then, the inverted relative fractal drum $(\mathbf{0},\Phi(\O))$ is also Minkowski measurable and we have$:$
\begin{equation}
\mathcal{M}^D(\mathbf{0},\Phi(\O))=-\frac{N+D}{N-D}\mathcal{M}^D(\ty,\O).
\end{equation}
\end{theorem}

\begin{proof}
Since, for a fixed $T>1$ from \cite[Theorem 3]{ra2} we have the equality
\begin{equation}\label{zeta_inv_}
\zeta_{\ty,\O}(s;T)=\zeta_{\mathbf{0},\Phi(\O)}(s;1/T),
\end{equation}
it is obvious that the relative distance zeta function of $(\mathbf{0},\Phi(\O))$ satisfies the analog of \cite[Theorem 6]{ra2} for classical relative fractal drums (see \cite[Theorem  4.1.14]{fzf}) and we have that
$$
\dim_B(\mathbf{0},\Phi(\O))=D(\zeta_{\mathbf{0},\Phi(\O)})=D(\zeta_{\ty,\O})=\dim_B(\ty,\O)=D.
$$
From the functional equation \eqref{zeta_inv_} we now conclude that $\zeta_{\mathbf{0},\Phi(\O)}$ satisfies the analog of Theorem  \ref{mink_suff} for classical relative fractal drums \cite[Theorem 5.4.2]{fzf} and hence, $(\mathbf{0},\Phi(\O))$ is Minkowski measurable.
Furthermore, $D$ is a simple pole and its residue is independent of $T$ which together with \cite[Theorem 6]{ra2} yields
\begin{equation}
\begin{aligned}
(N-D)\mathcal{M}^D(\mathbf{0},\Phi(\O))&=\res(\zeta_{\mathbf{0},\Phi(\O)},D)\\
&=\res(\zeta_{\ty,\O},D)=-(N+D)\mathcal{M}^D(\ty,\O),
\end{aligned}
\end{equation}
which yields the desired result.
\end{proof}

\begin{remark}
	It is clear that a reverse of Theorem \ref{inf_mink_con} can also be stated and proved, i.e., when one first imposes the analog assumptions on $(\mathbf{0},\O)$ and then concludes about $(\ty,\Phi(\O)$ but under the additional assumption that $\Phi(\O)$ has finite Lebesgue measure.
	We omit a detailed statement here.
\end{remark}


We will now state a version of Theorem~\ref{Mink_measurable_inf} dealing with a relative fractal drum $(\ty,\O)$ that is not Minkowski measurable, but its tube function satisfies a log-periodic asymptotic formula.
The theorem will demonstrate how the relative tube zeta function of $(\ty,\O)$ detects its `inner geometric oscillations' in terms of the principal complex dimensions of $(\ty,\O)$.
For a periodic function $G\colon\eR\to\eR$ with minimal period $T>0$, we define
\begin{equation}
G_0(\tau):=\chi_{[0,T]}(\tau)G(\tau)
\end{equation}
where $\chi_A$ is the characteristic function of a set $A$.
Furthermore we denote the Fourier transform of $G$ with $\{\mathfrak{F}G\}$ or $\hat{G}$, i.e.,
\begin{equation}
\{\mathfrak{F}G\}(s)=\hat{G}(s):=\int_{-\ty}^{+\ty}\mathrm{e}^{-2\pi\I s\tau}G(\tau)\di\tau.
\end{equation}

\begin{theorem}[Meromorphic extension - Minkowski measurable case]\label{Mink_nonmeasurable_inf}
Let $\O$ be a Lebesgue measurable subset of $\eR^N$ such that there exist $D<-N$, $\a>0$, and $G\colon\eR\to(0,+\ty)$, a nonconstant periodic function with period $T>0$, satisfying
\begin{equation}\label{periodic_mink_cont}
|_t\O|=t^{N+D}(G(\log t)+O(t^{-\a}))\quad\textrm{ as }t\to +\ty.
\end{equation}
Then $G$ is continuous, $\dim_B(\ty,\O)$ exists and $\dim_B(\ty,\O)=D$.
Furthermore, $\O$ is Minkowski nondegenerate at infinity with upper and lower Minkowski contents at infinity respectively given by
\begin{equation}
{\unb{\M}}^{D}(\ty,\O)=\min G,\quad {\ovb{\M}}^{D}(\ty,\O)=\max G.
\end{equation}
$($Hence, the range of $G|_{[0,T]}$ is equal to the interval $[{\unb{\M}}^{D}(\ty,\O),{\ovb{\M}}^{D}(\ty,\O)]$.$)$
Moreover, the tube zeta function $\widetilde{\zeta}_{\ty,\O}$ has for abscissa of convergence $D(\widetilde{\zeta}_{\ty,\O})=D$ and possesses a unique meromorphic extension $($still denoted by $\widetilde{\zeta}_{\ty,\O})$ to $($at least$)$ the half-plane $\{\re s>D-\a\}$.
In addition, the set of all the poles of $\widetilde{\zeta}_{\ty,\O}$ located in this half-plane is given by
\begin{equation}\label{tube_poles_at_inf}
\po_\a(\widetilde{\zeta}_{\ty,\O})=\left\{s_k=D+\frac{2\pi}{T}\I k\,:\,\hat{G}_0\left(\frac{k}{T}\right)\neq 0,\ k\in\Ze\right\};
\end{equation}
 they are all simple, and the residue at each $s_k\in\po_\a(\widetilde{\zeta}_{\ty,\O})$, $k\in\Ze$, is given by
\begin{equation}
\res(\widetilde{\zeta}_{\ty,\O},s_k)=\frac{1}{T}\hat{G}_0\left(\frac{k}{T}\right).
\end{equation}
By the reality principle, if $s_k\in\po_\a(\widetilde{\zeta}_{\ty,\O})$, then $s_{-k}\in\po_\a(\widetilde{\zeta}_{\ty,\O})$  and
\begin{equation}\label{R--L_lemma}
|\res(\widetilde{\zeta}_{\ty,\O},s_k)|\leq\frac{1}{T}\int_0^TG(\tau)\di\tau,\quad \lim_{k\to\pm\ty}\res(\widetilde{\zeta}_{\ty,\O},s_k)=0.
\end{equation}
Moreover, the set of poles $\po_{\a}(\widetilde{\zeta}_{\ty,\O})$ contains $s_0=D$, and
\begin{equation}\label{res=1/T}
\res(\widetilde{\zeta}_{\ty,\O},D)=\frac{1}{T}\int_0^TG(\tau)\di\tau.
\end{equation}
In particular, $\O$ is not Minkowski measurable at infinity and
\begin{equation}\label{res_mink_ineq}
{\unb{\M}}^{D}(\ty,\O)<\res(\widetilde{\zeta}_{\ty,\O},D)<{\ovb{\M}}^{D}(\ty,\O)<\ty.
\end{equation}

If, in addition, $G\in C^m(\eR)$ $($i.e., G is $m$ times continuously differentiable on $\eR)$ for some integer $m\geq 1$, and $G$ has an extremum $t_0$ such that
\begin{equation}\label{derivacije_G}
G'(t_0)=G''(t_0)=\ldots=G^{(m)}(t_0)=0,
\end{equation}
then there exists $C_m>0$ such that for all $k\in\Ze$ and $s_k\in\po_{\a}(\widetilde{\zeta}_{\ty,\O})$ we have
\begin{equation}\label{rezid_decay}
|\res(\widetilde{\zeta}_{\ty,\O},s_k)|\leq C_m|k|^{-m}.
\end{equation}
\end{theorem}

\begin{proof}
The fact that $G$ is continuous follows directly from \cite[Lemma 2.3.30]{fzf} by applying it to the function $F(t):=|_t\O|t^{N+D}$ which is defined and continuous for $t>0$.
We have that
$$
\begin{aligned}
\widetilde{\zeta}_{\ty,\O}(s)&=\int_P^{+\ty}t^{-s-N-1}|_t\O|\di t=\int_P^{+\ty}t^{-s-N-1}t^{N+D}(G(\log t)+O(t^{-\a}))\di t\\
&=\underbrace{\int_P^{+\ty}t^{D-s-1}G(\log t)\di t}_{\zeta_1(s)}+\underbrace{\int_{P}^{+\ty}t^{-s}O(t^{D-\a-1})\di t}_{{\zeta_2(s)}}
\end{aligned}
$$
for some fixed $P>0$.
As in the proof of Theorem~\ref{Mink_measurable_inf} we have that $D(\zeta_2)=D-\a$ and it suffices to prove that $\zeta_1$ can be meromorphically extended to the whole complex plane.
We will show this by finding a closed form for $\zeta_1$.
Since $G$ is $T$-periodic, we have that
$$
\zeta_1(s)=\int_P^{+\ty}t^{D-s-1}G(\log t+T)\di t.
$$
Let us introduce a new variable $u$ such that $\log u=\log t+T$, i.e., $u=\mathrm{e}^Tt$, to obtain
$$
\begin{aligned}
\zeta_1(s)&=\int_{\mathrm{e}^TP}^{+\ty}\mathrm{e}^{-T(D-s-1)}u^{D-s-1}G(\log u)\mathrm{e}^{-T}\di u\\
&=\mathrm{e}^{-T(D-s)}\int_{\mathrm{e}^TP}^{+\ty}u^{D-s-1}G(\log u)\di u\\
&=\mathrm{e}^{-T(D-s)}\left(\int_{P}^{+\ty}u^{D-s-1}G(\log u)\di u+\int_{\mathrm{e}^TP}^{P}u^{D-s-1}G(\log u)\di u\right)\\
&=\mathrm{e}^{-T(D-s)}\left(\zeta_1(s)+\int_{\mathrm{e}^TP}^{P}u^{D-s-1}G(\log u)\di u\right)
\end{aligned}
$$
which gives us a closed form for $\zeta_1$:
$$
\begin{aligned}
\zeta_1(s)&=\frac{\mathrm{e}^{-T(D-s)}}{\mathrm{e}^{-T(D-s)}-1}\int_{P}^{\mathrm{e}^TP}t^{D-s-1}G(\log t)\di t\\
&=\frac{\mathrm{e}^{T(s-D)}}{\mathrm{e}^{T(s-D)}-1}\underbrace{\int_{\log P}^{\log P +T}\mathrm{e}^{-\tau(s-D)}G(\tau)\di\tau}_{I(s)},
\end{aligned}
$$
where in the last equality we have introduced a new variable $\tau$ such that $\tau=\log t$.
The integral $I(s)$ is obviously an entire function since $P\neq 0,+\ty$.\footnote{By classical argument, see, e.g.,  \cite{Mattn}.}
From this we conclude that $\zeta_1$ is meromorphic on $\Ce$ and the set of its poles is equal to the set of solutions $s_k$ of $\exp(T(s-D))=1$ for which $I(s_k)\neq 0$.
If $I(s_k)=0$ then $s_k$ is a removable singularity of $\zeta_1$:
$$
\lim_{s\to s_k}\zeta_1(s)=\lim_{s\to s_k}\frac{s-s_k}{\mathrm{e}^{T(s-D)}-1}\,\mathrm{e}^{T(s-D)}\,\frac{I(s)}{s-s_k}=\frac{1}{P}I'(s_k)
$$
where $I'$ denotes the derivative of $I$.
Furthermore, since $\exp{(T(s_k-D))}=1$ we have that $\exp{(-\tau(s_k-D))}=\exp{(-2\pi\I k\tau/T)}$ and
\begin{equation}\label{I(s_k)}
\begin{aligned}
I(s_k)&=\int_{\log P}^{\log P +T}\mathrm{e}^{\frac{-2\pi\I k}{T}\tau}G(\tau)\di\tau\\
&=\int_{0}^{T}\mathrm{e}^{\frac{-2\pi\I k}{T}\tau}G(\tau)\di\tau=\hat{G}_0\left(\frac{k}{T}\right),
\end{aligned}
\end{equation}
where we have used the fact that both $\tau\mapsto G(\tau)$ and $\tau\mapsto\exp(-2\pi\I k\tau/T)$ are $T$-periodic.
This proves that the description of the poles of the tube zeta function of $\O$ at infinity that are contained in $\{\re s>D-\a\}$ is given by~\eqref{tube_poles_at_inf}.
Moreover, we observe that this set contains  $D$ since
\begin{equation}
I(D)=I(s_0)=\int_{0}^{T}G(\tau)\di\tau>0.
\end{equation}
Indeed, since $G$ is continuous and periodic we have from \eqref{periodic_mink_cont} that the range of $G|_{[0,T]}$ is equal to the whole interval $[{\unb{\M}}^{D}(\ty,\O),{\ovb{\M}}^{D}(\ty,\O)]$ and since $G$ is nonconstant, we deduce from~\eqref{periodic_mink_cont} that $0<{\unb{\M}}^{D}(\ty,\O)<{\ovb{\M}}^{D}(\ty,\O)<\ty$.
This induces that $D(\zeta_1)=D>D-\a=D(\zeta_2)$ from which it then follows immediately that $\widetilde{\zeta}_{\ty,\O}$ possesses a (unique) meromorphic extension to (at least) the half-plane $\{\re s>D-\a\}$.

Let us now compute the residues of $\widetilde{\zeta}_{\ty,\O}$ at $s_k=D+\frac{2\pi\I k}{T}$ for an arbitrary $k\in\Ze$, using~\eqref{I(s_k)} and L'Hospital's rule:
\begin{equation}\label{res_u_s_k}
\begin{aligned}
\res(\widetilde{\zeta}_{\ty,\O},s_k)&=\res(\zeta_1,s_k)\\
&=\lim_{s\to s_k}\frac{s-s_k}{\mathrm{e}^{T(s-D)}-1}\,\mathrm{e}^{T(s-D)}I(s)=\frac{1}{T}\hat{G}_0\left(\frac{k}{T}\right).
\end{aligned}
\end{equation}
Substituting $k=0$ in the above expression we obtain that~\eqref{res=1/T} which, in turn, implies the inequalities in~\eqref{res_mink_ineq}.

Furthermore, as it is well-known, since $G_0\in L^1(\eR)$, we have $|\hat G_0(\tau)|\le\|G_0\|_{L^1(\eR)}=\|G\|_{L^1(0,T)}$ and $\lim_{t\to\ty}\hat G_0(t)=0$ (by the 
Riemann--Lebesgue lemma; see, e.g., \cite{rudin} or \cite[p.\ 101]{mizu}),
so that \eqref{R--L_lemma} follows immediately from~\eqref{res_u_s_k}.

If the function $G$ is of class $C^m$, this does not imply that $G_0$ is of the same class.
However, we can define $G_1:\eR\to\eR$ by
\begin{equation}
G_1(\tau)=
\begin{cases}
G(\tau)-{\unb{\M}}^{D}(\ty,\O)& \mbox{if $\tau\in[0,T]$,}\\
0& \mbox{if $\tau\notin[0,T].$}
\end{cases}
\end{equation}
Since the value of ${\unb{\M}}^{D}(\ty,\O)$ is in the range of $G$, we may assume without loss of generality that $t_0=0$ is a minimum of $G$; namely,  $G(0)=G(T)={\unb{\M}}^{D}(\ty,\O)$.
If that is not the case, one can translate the graph of $G$ in the horizontal direction in order to achieve this.
Furthermore, ${\unb{\M}}^{D}(\ty,\O)$ is equal to the minimal value of $G$; hence, $G_1(0)=G_1(T)=0$.
This implies that $G_1$ is continuous on $\eR$, and moreover, 
from~\eqref{derivacije_G}, we have that $G_1$ has the same regularity as $G$; that is,
$G_1\in C^m(\eR)$.
A direct computation shows that for each $t\in\eR$,
\begin{equation}
\hat G_1(t)=\hat G_0(t)-{\unb{\M}}^{D}(\ty,\O)\frac{1-{\mathrm{e}}^{-2\pi {\I}t\cdot T}}{2\pi {\I} t},
\end{equation}
from which it follows that
\begin{equation}
\res(\widetilde{\zeta}_{\ty,\O},s_k)=\frac1T\hat G_0\left(\frac kT\right)=\frac1T\hat G_1\left(\frac kT\right).
\end{equation}
Since $G_1\in C^{m}(\eR)$, a standard result from Fourier analysis (see e.g.\ \cite[p.\ 103]{mizu}) yields a constant $C_m>0$ such that $|\hat G_1(t)|\le C_m t^{-m}$ for all $t\in\eR$, which proves~\eqref{rezid_decay}.
We point out that the same conclusion can be achieved by defining $G_1(\tau)=G(\tau)-{\ovb{\M}}^{D}(\ty,\O)$.
\end{proof}

\section{Properties of fractal zeta functions at infinity}\label{prop_sec_ty}

In this section we will prove some useful properties of the distance and tube zeta functions at infinity.
We will start the section with a useful lemma from which we will derive the scaling property of fractal zeta functions at infinity.
Recall that for a parameter $\lambda>0$ and a subset $\O$ of $\eR^N$ we define
\begin{equation}
\lambda\O:=\{\lambda x\,:\,x\in\O\}.
\end{equation}

\begin{lemma}\label{skaliranje_lema}
Let $\O$ be a Lebesgue measurable subset of $\eR^N$ of finite Lebesgue measure. For any $\lambda>0$ and $t>0$ we have$:$
\begin{equation}\label{skaliranje}
|B_{t}(0)^c\cap\lambda\O)|=\lambda^{N}|B_{t/\lambda}(0)^c\cap\O|
\end{equation}
and
\begin{equation}
{{\ovb{\M}}}^{r}(\ty,\lambda\O)=\lambda^{-r}{{\ovb{\M}}}^{r}(\ty,\O),\quad{\unb{{\M}}}^{r}(\ty,\lambda\O)=\lambda^{-r}{\unb{{\M}}}^{r}(\ty,\O),
\end{equation}
for every real number $r$.
\end{lemma}

\begin{proof}
We have that $\lambda(B_{t/\lambda}(0)^c\cap\O)=B_{t}(0)^c\cap\lambda\O$ from which the first part of the lemma follows directly.
For the second part, we observe that
$$
{{\ovb{\M}}}^{r}(\ty,\lambda\O)=\limsup_{t\to +\ty}\frac{|B_{t}(0)^c\cap\lambda\O|}{t^{N+r}}=\limsup_{t\to +\ty}\frac{\lambda^{-r}|B_{t/\lambda}(0)^c\cap\O|}{(t/\lambda)^{N+r}}=\lambda^{-r}{\ovb{\M}}^{r}(\ty,\O)
$$
and similarly for the lower limit which concludes the proof of the lemma.
\end{proof}

The next result is a scaling property of the distance zeta function at infinity which will prove useful in examples, as well as in the construction of quasiperiodic sets at infinity.

\begin{proposition}[Scaling property of the distance zeta function at infinity]\label{scaling_prop}
Let $\O$ be a Lebesgue measurable subset of $\eR^N$ with finite Lebesgue measure, $T>0$ and $\lambda>0$.
Then we have $D(\zeta_{\ty,\lambda\O}=D(\zeta_{\ty,\O})=\ovb{\dim}_B(\ty,\O)$ and
\begin{equation}\label{scaling}
\zeta_{\ty,\g\O}(s;\g T)=\g^{-s}\zeta_{\ty,\O}(s;T)
\end{equation}
for all $s\in\Ce$ with $\re s>\ovb{\dim}_B(\ty,\O)$.
Furthermore, if $\omega$ is a simple pole of a meromorphic extension of $\zeta_{\ty,\O}$ to some open connected neighborhood of the critical line $\{\re s=\ovb{\dim}_B(\ty,\O)\}$, 
then
\begin{equation}\label{res_skal}
\res(\zeta_{\ty,\g\O},\omega)=\g^{-\omega}\res(\zeta_{\ty,\O},\omega).
\end{equation}
\end{proposition}

\begin{proof}
From Lemma~\ref{skaliranje_lema} we know that $\ovb{\dim}_B(\ty,\g\O)=\ovb{\dim}_B(\ty,\O)$.
We will prove the scaling formula~\eqref{scaling} by introducing a new variable $y=x/\g$ and using the change of variables formula for the Lebesgue integral.
Noting that $\di x=\g^N\di y$, we have
\begin{equation}
\begin{aligned}
\zeta_{\ty,\g\O}(s;\g T)&=\int_{B_{\g T}(0)\cap\g\O}|x|^{-s-N}\di x\\
&=\int_{B_{T}(0)\cap\O}|\g y|^{-s-N}\g^N\di y\\
&=\g^{-s}\int_{B_{T}(0)\cap\O}|y|^{-s-N}\di y=\g^{-s}\zeta_{\ty,\O}(s;T)
\end{aligned}
\end{equation}
for $s\in\Ce$ such that $\re s>\ovb{\dim}_B(\ty,\O)$.

Note that by the principle of analytic continuation, if one of the two zeta functions in~\eqref{scaling} has a meromorphic extension to some open connected neighborhood $U$ of the critical line, then so does the other and~\eqref{scaling} is still valid for $s\in U$.
Furthermore, if that is the case and $\omega\in U$ is a simple pole of $\zeta_{\ty,\O}$, then we have that
$$
\g^{-s}(s-\omega)\zeta_{\ty,\O}(s;T)=(s-\omega)\zeta_{\ty,\g\O}(s;\g T)
$$
holds on a pointed neighborhood of $\omega$.
Finally, since the value of the residue of the distance zeta function at infinity does not depend on $T$ we get~\eqref{res_skal} by letting $s\to\omega$, $s\neq\omega$.
\end{proof}
 
We will now prove a result which will be very useful for almost all examples that we will look at.
Namely, as before, for a set $\O$ in $\eR^N$ with finite Lebesgue measure it will be easier to calculate (as we have already done several times) a closed form for the corresponding distance zeta function at infinity by using the max norm in $\eR^N$ instead of the usual Euclidean norm.



\begin{definition}\label{hold_norm}
Let $\|\cdot\|_1$ and $\|\cdot\|_2$ be two $($necessarily equivalent$)$ norms on $\eR^N$ and let $\O\subseteq\eR^N$.
We will say that $\|\cdot\|_1$ and $\|\cdot\|_2$ are {\em equivalent of order $\a\in\eR$ for $(\ty,\O)$} if
\begin{equation}\label{norm_asym}
\|x\|_1=\|x\|_2+O\left(\|x\|_1^{\a}\right),\quad\mathrm{as}\quad \|x\|_1\to+\ty,\ x\in\O.
\end{equation}
In this case we will write
\begin{equation}
\|\cdot\|_1\underset{(\ty,\O)}{\overset{\a}{\sim}} \|\cdot\|_2.
\end{equation}
\end{definition}

This equivalence is well defined since the two norms are equivalent in the standard sense.
More precisely, since there exist $m,M>0$ such that $m\|\cdot\|_1\leq\|\cdot\|_2\leq M\|\cdot\|_1$ we have that $O\left(\|x\|_1^{\a}\right)=O\left(\|x\|_2^{\a}\right)$ for every $\a\in\eR$ when $\|x\|_1\to +\ty$ or $\|x\|_2\to +\ty$.
From this, one gets symmetry and transitivity easily.

\begin{theorem}[Behavior under change of norm]\label{equiv_mero}
Let $\O$ be a Lebesgue measurable subset of $\eR^N\setminus\{0\}$ with finite Lebesgue measure and assume $\ovb{D}:=\ovb{\dim}_B(\ty,\O)<-N$.
Furthermore, assume that $\|\cdot\|$ is a norm in $\eR^N$ such that for some $\a\in(-\ty,1)$ we have 
\begin{equation}
|x|\underset{(\ty,\O)}{\overset{\a}{\sim}}\|x\|.
\end{equation}
Then, the difference
\begin{equation}
{\zeta}_{\ty,\O}(\,\cdot\,)-{\zeta}_{\ty,\O}(\,\cdot\,;\|\cdot\|)
\end{equation}
is holomorphic on $($at least$)$ the half-plane
\begin{equation}\label{podskup}
\{\re s>(\ovb{D}-(1-\a))\}.
\end{equation}
\end{theorem}

\begin{proof}
We observe that for every $s\in\Ce$ the function $f_s(z):=z^{-s-N}$ is holomorphic on $\Ce\setminus\{0\}$ and define $F(s,x):=f_s(|x|)-f_s(\|x\|)$.
Then, from Corollary~\ref{comp_ocj} applied to $f_s$, we conclude that there exists a function $r\colon\Ce\times({\O}\setminus X)\to(0,+\ty)$ such that
\begin{equation}\label{eksp}
|F(s,x)|=\left||x|^{-s-N}-\|x\|^{-s-N}\right|\leq\sqrt{2}|s+N|r(s,x)^{-\re s-N-1}\big||x|-\|x\|\big|,
\end{equation}
where $X:=\{x\in\Omega:|x|=\|x\|\}$.
Let $m$ and $M$ be the positive constants such that $m|x|\leq\|x\|\leq M|x|$ for $x\in\eR^N$ and denote
\begin{equation}
C_m:=\min\{1,m\},\quad C_M:=\max\{1,M\}.
\end{equation}
Furthermore, since
\begin{equation}
|x|<r(s,x)<\|x\|\quad\mathrm{or}\quad \|x\|<r(s,x)<|x|
\end{equation}
we have that
\begin{equation}
C_m|x|\leq r(s,x)\leq C_M|x|,
\end{equation}
which implies that
\begin{equation}\label{ocjw}
r(s,x)^{-\re s-N-1}\leq|x|^{-\re s-N-1}\max\{C_m^{-\re s-N-1},C_M^{-\re s-N-1}\}.
\end{equation}
Furthermore, by taking $T>1$ sufficiently large, we can assume that $\big||x|-\|x\|\big|\leq c|x|^{\alpha}$ which together with~\eqref{ocjw} and~\eqref{eksp} yields
\begin{equation}
|F(s,x)|\leq c\sqrt{2}|s+N|\max\{C_m^{-\re s-N-1},C_M^{-\re s-N-1}\}|x|^{-\re s-N-1+\a}.
\end{equation}
Suppose now that $K$ is a compact subset in $\{\re s>\ovb{D}-(1-\a)\}$, and let
\begin{equation}
C_K:=\max_{s\in K}\big\{c\sqrt{2}|s+N|\max\{C_m^{-\re s-N-1},C_M^{-\re s-N-1}\}\big\}
\end{equation}
and define the function $g_K$ as follows:
\begin{equation}
g_K(x):=C_K|x|^{-(\min\{\re s\,:\,s\in K\}-\a+1)-N}
\end{equation}
so that we have $|F(s,x)_{|K}|\leq g_K(x)$ for $x\in{_T\O}\setminus X$.
We observe that $g_K$ is in $L^1({_T\O})$, since if $s\in K$, then $\re s>\ovb{D}-(1-\a)$ so that
$$
\min\{\re s\,:\,s\in K\}-\a+1>\ovb{D}-(1-\a)-\a+1=\ovb{D},
$$
which, in turn, implies that
$$
\int_{_T\O}g_K(x)\di x=C_K\,\zeta_{\ty}(\min\{\re s\,:\,s\in K\}-\a+1;\O)<\ty.
$$

Finally, we conclude that $F(s,x)$ satisfies the hypotheses of   \cite[Theorem 2.1.47 and Remark 2.1.48]{fzf} (see also \cite{Mattn}) and therefore
$$
\int_{_T\O\setminus X}F(s,x)\di x={\zeta}_{\ty,\O\setminus X}(s)-{\zeta}_{\ty,\O\setminus X}(s;\|\cdot\|)
$$
is holomorphic on $\{\re s>\ovb{D}-(1-\a)\}$.
On the other hand, ${\zeta}_{\ty,X}\equiv{\zeta}_{\ty,X}(s;\|\cdot\|)$ and one also has that $\zeta_{\ty,\O}(s)=\zeta_{\ty,\O\setminus X}(s)+\zeta_{\ty, X}(s)$ and analogously for ${\zeta}_{\ty,\O}(s;\|\cdot\|)$.
In light of this observation the proof of the theorem is now complete.
\end{proof}

\begin{corollary}
Let $\O$ be a measurable subset in $\eR^N$ with $|\O|<\ty$ such that ${\dim}_B(\ty,\O)={D}$ exists.
Furthermore, assume that the distance zeta function of $\O$ at infinity can be meromorphically extended to an open connected neighborhood $U$ of the closed half-plane $\{\re s\geq D\}$.
Let $\|\,.\,\|$ be another norm in $\eR^N$ such that $|x|\underset{(\ty,\O)}{\overset{\a}{\sim}}\|x\|$ for some $\a\in(-\ty,1)$.
Then $\widetilde{\zeta}_{\ty,\O}(\,\cdot\,;\|\,.\,\|)$ can be meromorphically extended to $($at least$)$ $V:=U\cap\{\re s>D-(1-\a)\}$.
Furthermore, the sets of poles in $V$ of the two zeta functions, together with their multiplicities, coincide.
Moreover, the principal parts of the Laurent expansion around each pole in $V$ also coincide.
In particular, if $\omega$ is a simple pole in $V$, then
\begin{equation}
\res(\widetilde{\zeta}_{\ty,\O},\omega)=\res(\widetilde{\zeta}_{\ty,\O}(\,\cdot\,;\|\,.\,\|),\omega).
\end{equation}
\end{corollary}

\begin{proof}
Since, by hypothesis, $\widetilde{\zeta}_{\ty,\O}$ is meromorphic on $V=U\cap\{\re s>D-(1-\a)\}$ the corollary follows directly from Theorem~\ref{equiv_mero} which states that the difference of these two distance zeta functions is holomorphic on $V$.
\end{proof}

\begin{remark}
It is clear that the above corollary is still valid if we interchange the roles of the two distance zeta functions.
\end{remark}

An important special case of the above theorem, which we will be using in almost all examples considered, is when the set $\O\subseteq\eR^N$ is contained in a cylinder of finite radius.
This is in fact \cite[Proposition 6]{ra2} for which the proof was omitted in \cite{ra2}.
We restate it here and give a short proof by using Theorem \ref{equiv_mero}.

\begin{proposition}\label{euc_ty}
Let $\O\subseteq\eR^N$ with $|\O|<\ty$ be such that it is contained in a cylinder
\begin{equation}
x_2^2+x_3^2+\cdots+x_N^2\leq C
\end{equation}
for some constant $C>0$ where $x=(x_1,\ldots,x_N)$.
Furthermore, let $\ovb{D}:=\ovb{\dim}_B(\ty,\O)$ and $T>0$.
Then
\begin{equation}
\zeta_{\ty,\O}(s;T)-\zeta_{\ty,\O}(s;T;|\cdot|_{\infty})
\end{equation}
is holomorphic on $($at least$)$ the half-plane $\{\re s>\ovb{D}-2\}$.

Furthermore, if any of the two distance zeta functions possesses a meromorphic extension to some open connected neighborhood $U$ of the critical line $\{\re s=\ovb{D}\}$, then the other one possesses a meromorphic extension to $($at least$)$ $V:=U\cap\{\re s>\ovb{D}-2\}$.

Moreover, their multisets of poles in $U\cap\{\re s>\ovb{D}-2\}$ coincide.
\end{proposition}

\begin{proof}
We observe that for $T>0$ sufficiently large we have
$$
|x|-|x|_{\ty}=|x|-|x_1|=\frac{\sum_{i=2}^Nx_i^2}{|x|+|x_1|}\leq C|x|^{-1},\quad x\in{_T\O}.
$$
In other words $|x|\underset{(\ty,\O)}{\overset{-1}{\sim}}\|x\|$ and the conclusion now follows by applying Theorem~\ref{equiv_mero}.
\end{proof}

\section{Maximally hyperfractal and quasiperiodic sets at infinity}\label{qp_sets}

In this section we will construct quasiperiodic subsets of $\eR^2$ with prescribed box dimension at infinity.
We will use the Cantor-like two parameter sets $\O_{\ty}^{(a,b)}$ introduced in \cite{ra2} which will be our building blocks for the construction of a {\em maximally hyperfractal set at infinity}; that is, according to the terminology of~\cite{fzf}, a set $\O$ with its distance zeta function at infinity having the critical line $\{\re s=\ovb{\dim}_B(\ty,\O)\}$ as a natural boundary.
This construction will also give examples of algebraically and transcendentally quasiperiodic sets at infinity by using some classical results from transcendental number theory.

One of the open problems in~\cite{fzf} was the question of existence of algebraically quasiperiodic bounded sets and relative fractal drums.
The results of this section give a positive answer in the case of relative fractal drums of type $(\ty,\O)$.
Furthermore, a similar construction can be performed in the context of standard relative fractal drums.
More precisely, one can take the inverted relative fractal drum $(\mathbf{0},\Phi(\O))$ where $\O$ is the quasiperiodic set at infinity constructed here.
The distance zeta functions of $(\ty,\O)$ and $(\mathbf{0},\Phi(\O))$ are essentially the same by \cite[Theorem 3]{ra2}.
On the other hand, one should still check directly the condition of quasiperiodicity of the relative fractal drum $(\mathbf{0},\Phi(\O))$ since we do not have a direct way of doing it via the geometric inversion.
The reason for this is in the fact that we would need an asymptotic formula which relates the tube function $t\mapsto|B_t(0)^c\cap\O|$ of $\O$ at infinity and the relative tube function $t\mapsto|B_{1/t}(0)\cap\Phi(\O)|$ as $t\to +\ty$.
We do not know if such formula can be derived in the general case, but we still conjecture that $(\mathbf{0},\Phi(\O))$ will be a quasiperiodic relative fractal drum with the same quasiperiods as $(\ty,\O)$.

Another idea to construct an algebraically quasiperiodic relative fractal drum $(A,\O)$ is to use the geometric inversion in one coordinate; that is, $\Phi_1(x,y):=(1/x,y)$ and apply it to the quasiperiodic relative fractal drum $(\ty,\O)$ constructed here.
We leave this, as well as other properties of the `partial geometric inversion' for future work.
One more approach would be to consider a ``radial version'' of RFDs considered here since then the geometric inversion of such an RFD would be much more easier to handle.

We now recall the Cantor-like sets at infinity from \cite[Section 5]{ra2}.
Namely, one ``stacks'' the translated images of the two-parameter unbounded sets $\O_{m}^{(a,b)}$ along the $y$-axis on top of each other, where
$$
\O_{m}^{(a,b)}:=\{(x,y)\in\eR^2\,:\,x>a^{-m},\ 0<y<x^{-b}\},\quad m\geq 1.
$$
More precisely, for each $m\geq 1$ one takes $2^{m-1}$ copies of $\O_{m}^{(a,b)}$ and arranges all of them by vertical translations so that they are pairwise disjoint and lie in the strip $\{0\leq y\leq S\}$ where $S$ denotes the finite sum of widths of all of these sets.
Under the condition that $b>1+\log_{1/a}2>$ the disjoint union of all these sets, $\O_{\ty}^{(a,b)}$ is of finite volume and also $S$ is finite; see \cite{ra2} and Figure \ref{Cantor_u_besk}.

\begin{figure}[h]
\begin{center}
\includegraphics[width=12cm,height=6cm]{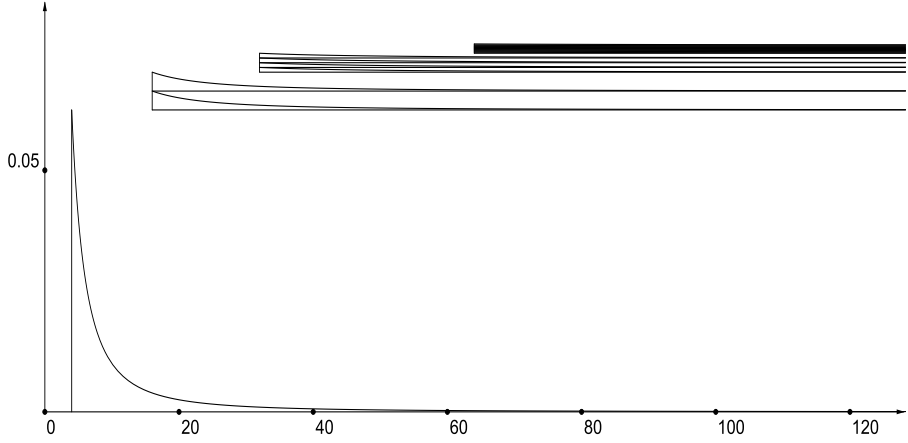}
\end{center}
\caption{The two parameter set $\O_{\ty}^{(a,b)}$ with $a=1/4$ and $b=2$. The axes are not in the same scale and only the first four steps of the set $\O_{\ty}^{(1/4,2)}$ are depicted; that is, for $m=1,2,3,4$.}
\label{Cantor_u_besk}
\end{figure}

In \cite[Example 6]{ra2}, the distance zeta function in terms of the $|\cdot|_{\ty}$ norm on $\mathbb{R}^2$ has been calculated and shown to be meromorphic in all of $\Ce$ and given by
\begin{equation}
\zeta_{\ty,\O_{\ty}^{(a,b)}}(s;|\cdot|_{\ty})=\frac{1}{s+b+1}\cdot\frac{1}{a^{-(s+b+1)}-2}.
\end{equation}
Furthermore, it was shown that the set of complex dimensions of $\O_{\ty}^{(a,b)}$ at infinity visible through $W:=\{\re s>\log_{1/a}-b-3\}$ is given by
\begin{equation}\label{twop_set}
\{-(b+1)\}\cup\left(\log_{1/a}2-(b+1)+\frac{2\pi}{\log (1/a)}\I\Ze\right)
\end{equation}
and from which it follows that
\begin{equation}
{{\ovb\dim}_{B}}(\ty,\O_{\ty}^{(a,b)})=\log_{1/a}2-(b+1).
\end{equation}

\begin{remark}
The {\em oscillatory period} of $\O_{\ty}^{(a,b)}$ is equal to $\mathbf{p}(a)=2\pi/\log(1/a)$ and note that that $\mathbf{p}(a)\to 0$ as $a\to 0^+$.
\end{remark}

\begin{example}\label{ex:2param}
We will compute the box dimension of $\O_{\ty}^{(a,b)}$ at infinity directly.
For the calculation we will measure the neighborhoods of infinity in the $|\cdot|_\ty$ norm.
As $\O_{\ty}^{(a,b)}$ is contained in a horizontal strip of finite width, according to \cite[Lemma 3]{ra2}, this will not affect the value of the Minkowski content of $\O_{\ty}^{(a,b)}$ at infinity; see also Lemma \ref{tube_functions} in the Appendix.
Now, for $t>1/a$ we have
$$
\begin{aligned}
|K_t(0)^c\cap\O_{\ty}^{(a,b)}|&=\sum_{n=1}^{\lfloor\log_{1/a}t\rfloor}2^{n-1}\int_t^{+\ty}x^{-b}\di x+\sum_{n>\lfloor\log_{1/a}t\rfloor}^{\ty}2^{n-1}\int_{a^{-n}}^{+\ty}x^{-b}\di x\\
&=\frac{1}{b-1}\left[t^{1-b}\sum_{n=1}^{\lfloor\log_{1/a}t\rfloor}2^{n-1}+\sum_{n>\lfloor\log_{1/a}t\rfloor}^{\ty}2^{n-1}(a^{b-1})^n\right]\\
&=\frac{1}{b-1}\left[t^{1-b}\left(2^{\lfloor\log_{1/a}t\rfloor}-1\right)+\frac{1}{a^{1-b}-2}\cdot 2^{\lfloor\log_{1/a}t\rfloor}a^{(b-1)\lfloor\log_{1/a}t\rfloor}\right]. 
\end{aligned}
$$ 
Using the fact that $\lfloor\log_{1/a}t\rfloor=\log_{1/a}t-\{\log_{1/a}t\}$ and $2^{\log_{1/a}t}=t^{\log_{1/a}2}$, we then have that
$$
\begin{aligned}
|K_t(0)^c\cap\O_{\ty}^{(a,b)}|&=\frac{t^{1-b+\log_{1/a}2}}{b-1}\left[2^{-\{\log_{1/a}t\}}+\frac{1}{a^{1-b}-2}\cdot\left(\frac{a^{1-b}}{2}\right)^{\{\log_{1/a}t\}}\right]-\frac{t^{1-b}}{b-1}.
\end{aligned}
$$
From this we deduce that for $D:=\log_{1/a}2-(b+1)$ we have
\begin{equation}
|K_t(0)^c\cap\O_{\ty}^{(a,b)}|=t^{2+D}\left(G(\log t)-\frac{t^{-\log_{1/a}2}}{b-1}\right)\quad\textrm{ as }t\to +\ty
\end{equation}
with $G$ being the $T$-periodic function
\begin{equation}\label{G_tau}
G(\tau):=\frac{2^{-\left\{\frac{\tau}{\log(1/a)}\right\}}}{b-1}\left(1+\frac{(a^{1-b})^{\left\{\frac{\tau}{\log(1/a)}\right\}}}{a^{1-b}-2}\right),
\end{equation}
where $T:=\log(1/a)$.
Furthermore, this result implies that
$$
{\dim}_B(\ty,\O_{\ty}^{(a,b)})=\log_{1/a}2-(b+1).
$$
Note that ${{\dim}}_B({\ty},\O_{\ty}^{(a,b)})\to -\ty$ as $b\to +\ty$ and ${{\dim}}_B({\ty},\O_{\ty}^{(a,b)})\to -(b+1)<-2$ as $a\to 0^+$ but can be made as close to $-2$ as desirable. 
Moreover, $\O_{\ty}^{(a,b)}$ is not Minkowski measurable at infinity but it is Minkowski nondegenerate with
\begin{equation}\label{G_sadrzaj_ab_Omega}
{\ovb{\M}}^{D}(\ty,\O_{\ty}^{(a,b)})=\max G=G(0)=\frac{1}{b-1}\cdot\frac{a^{1-b}-1}{a^{1-b}-2}
\end{equation}
and
\begin{equation}\label{D_sadrzaj_ab_Omega}
{\unb{\M}}^{D}(\ty,\O_{\ty}^{(a,b)})=\min G=G(\tau_{\mathrm{min}}),
\end{equation}
where $\tau_{{\mathrm{min}}}$ is the unique point of the global minimum of the function $G$ on the interval $[0,1]$ which can be explicitly computed:
$$
\tau_{{\mathrm{min}}}=\frac{\log(1+(b-1)\log_2a)-\log(2-a^{1-b})}{(b-1)\log a}.
$$
\end{example}

In the next theorem we will construct a maximal hyperfractal set $\O$ at infinity.
More precisely, we will now construct a set with a prescribed Minkowski dimension $D<-2$ at infinity such that every point on the abscissa of convergence $\{\re s=D\}$ is a nonremovable singularity of its zeta function at infinity.
In accordance with the definitions introduced in~\cite{fzf} in the case of relative fractal drums, will call such sets {maximally hyperfractal at infinity}.

\begin{theorem}[Maximally hyperfractal set at infinity]\label{qp_construction}
For any $D<-2$ there exists a set $\O\subseteq\eR^2$ of finite Lebesgue measure such that it is maximally hyperfractal with $\dim_B(\ty,\O)=D$ and Minkowski nondegenerate at infinity.
\end{theorem}
 
\begin{proof}
Let us fix $D<-2$ and choose a nonincreasing sequence $(a_n)_{n\geq 1}$ such that $0<a_n<1/2$ for every $n\in\eN$ and $a_n\to 0^+$ as $n\to +\infty$.
Furthermore, we define the sequence $b_n:=\log_{1/a_n}2-D-1$ and observe that for $D<-2$ the condition $b_n>1+\log_{1/a_n}2$ is fulfilled.
For the two parameter unbounded set $\O_\ty^{(a_n,b_n)}$ from Example \ref{ex:2param} we have that ${\dim}_B({\ty},\O_{\ty}^{(a_n,b_n)})=D$.
The next step is to scale every one of these sets with a suitable parameter, namely we define the sets $\widetilde{\O}_n$ for every $n\in\eN$ as follows:
$$
\widetilde{\O}_{n}:=\frac{1}{2^n}\O_{\ty}^{(a_n,b_n)}.
$$
Finally we construct the sets $\O_n$ by translating each set $\widetilde{\O}_n$ vertically for the amount $l_n$ which is equal to the sum of the heights of each $\widetilde{\O}_k$ for $k<n$, i.e., $l_1:=0$ and
$$
l_n:=\sum_{k=1}^{n-1}\frac{1}{2^k}\frac{a_k^{b_k}}{1-2a_k^{b_k}}
$$
for $n>1$ and define the set $\O$ to be the disjoint union of the sets $\O_n$.
Now we observe that the scaling factor in the definition of the sets $\widetilde{\O}$ ensures that the set $\O$ has finite Lebesgue measure and that it lies in a horizontal strip of finite width.

Similarly as before, this ensures us that calculating the tube formula of $\O$ using the $|\cdot|_{\ty}$-norm on $\eR^2$ will not affect the values of the upper and lower Minkowski contents of $\O$ at infinity.
For $t>1$ we have that
$$
\begin{aligned}
|K_t(0)^c\cap\O|&=\sum_{n=1}^\ty|K_t(0)^c\cap\O_n|=\sum_{n=1}^\ty|K_t(0)^c\cap 2^{-n}\O_{\ty}^{(a_n,b_n)}|\\
&=\sum_{n=1}^{\ty}2^{-2n}|K_{2^nt}(0)^c\cap\O_{\ty}^{(a_n,b_n)}|\\
&=\sum_{n=1}^{\ty}\frac{t^{2+D}}{2^{-Dn}}\left(G_n\left(\log(2^nt)\right)-\frac{t^{-\log_{1/a_n}2}}{2^{n\log_{1/a_n}2}(b_n-1)}\right)
\end{aligned}
$$
where we have used~\eqref{skaliranje} with $N=2$ and $G_n$ is the $\log(1/a_n)$-periodic function defined by~\eqref{G_tau} with $a$ and $b$ replaced by $a_n$ and $b_n$ respectively.
In other words, we have:
\begin{equation}\label{K_tcapO}
|K_t(0)^c\cap\O|=t^{2+D}\left(G\left(\log{t}\right)-\sum_{n=1}^{\ty}\frac{t^{-\log_{1/a_n}2}}{(b_n-1)2^{n(-D+\log_{1/a_n}2)}}\right),
\end{equation}
where
\begin{equation}
G(\tau):=\sum_{n=1}^{\ty}2^{nD}G_n\left(\tau+n\log{2}\right).
\end{equation}
The convergence of the sum for every $t>1$ in~\eqref{K_tcapO} follows from the facts that $\log_{1/a_n}2\in(0,1)$, $-D>2$ and $b_n-1>-D-2>0$ for all $n\in\eN$, i.e.,
$$
\sum_{n=1}^{\ty}\frac{t^{-\log_{1/a_n}2}}{(b_n-1)2^{n(-D+\log_{1/a_n}2)}}\leq -\frac{1}{D+2}\sum_{n=1}^{\ty}\frac{1}{(2^{-D+\log_{1/a_n}2})^n}<\ty.
$$
Furthermore, the series defining the function $G$ is also convergent for $\tau>0$.
To see this, we observe that from~\eqref{G_sadrzaj_ab_Omega} we have:
$$
G_n(\tau)\leq{\ovb{\M}}_{\ty}^{D}(\O_{\ty}^{(a_n,b_n)})=\frac{1}{b_n-1}\frac{a_n^{1-b_n}-1}{a_n^{1-b_n}-2}\leq 1
$$
for all $n\in\eN$.
The last inequality above can be easily shown from the conditions on $a_n$ and $b_n$.
Furthermore, from this we conclude that
$$
G(\tau)=\sum_{n=1}^{\ty}2^{nD}G_n\left(\tau+n\log{2}\right)\leq\sum_{n=1}^{\ty}\frac{1}{(2^{-D})^n}<\ty.
$$
In particular,
$$
{\ovb{\M}}^{D}(\ty,\O)\leq\sum_{n=1}^{\ty}\frac{1}{(2^{-D})^n}<\ty.
$$
On the other hand, for the lower Minkowski content of $\O$ at infinity we can use the fact that $\O\supseteq\O_1$ and therefore
$$
{\unb{\M}}^{D}(\ty,\O)\geq{\unb{\M}}^D(\ty,\O_1)={\unb{\M}}^D(\ty,2^{-1}\O_{\ty}^{(a_1,b_1)})=2^{D}{\unb{\M}}^D(\ty,\O_{\ty}^{(a_1,b_1)})>0.
$$
The last equality above is a consequence of Lemma~\ref{skaliranje_lema} with $r=D$, while the conclusion of positivity follows from~\eqref{D_sadrzaj_ab_Omega}.

Let us now show that for the distance zeta function of $\O$ at infinity the critical line $\{\re s=D\}$ is a natural boundary.
Using the scaling property of the distance zeta function at infinity from Proposition~\ref{scaling_prop} we have that
$$
\zeta_{\ty,\O}(s)=\sum_{n=1}^{\ty}\zeta_{\ty,2^{-n}\O_{\ty}^{(a_n,b_n)}}(s;1)=\sum_{n=1}^{\ty}2^{ns}\cdot\zeta_{\ty,\O_{\ty}^{(a_n,b_n)}}(s;2^n)
$$
and it is holomorphic on $\{\re s>D\}$.
Furthermore, according to \cite[Example 6]{ra2}, for every $n\in\eN$ the zeta function $\zeta_{\ty,\O_{\ty}^{(a_n,b_n)}}(s;2^n;|\cdot|_{\ty})$ is meromorphic on $\Ce$ and has simple poles at $D+\frac{2\pi\I}{\log(1/a_n)}\Ze$.
Since $\O_{\ty}^{(a_n,b_n)}$ is contained in a strip, i,e, a cylinder of finite width, according to Proposition~\ref{euc_ty}, we have that $\zeta_{\ty,\O}(s;2^n)$ is meromorphic at least on $\{\re s>D-2\}$ and its poles in that half-plain coincide with that of $\zeta_{\ty,\O_{\ty}^{(a_n,b_n)}}(s;2^n;|\cdot|_{\ty})$.
From this we conclude that the set of poles of $\zeta_{\ty,\O}(s)$ is dense in the critical line $\{\re s=D\}$ since $\log(1/a_n)\to +\ty$ as $n\to +\ty$.
This, in turn, implies that every point of the critical line is a nonremovable singularity of $\zeta_{\ty,\O}(s)$, i.e., $\O$ is maximally hyperfractal at infinity.
\end{proof}

Now we are ready to show that a careful choice of the parameters in the two parameter set from Example gives examples of algebraically and transcendentally quasiperiodic sets at infinity.
First we recall some needed definitions from number theory for the convenience of the reader.

The {\em field of algebraic numbers} (often denoted by $\ov\Qu$\label{ovQ} in the literature) 
can be viewed (up to isomorphism) as the algebraic closure of $\Qu$ (the field of rational numbers)
and is obtained by adjoining to $\Qu$ the roots of the polynomial equations with coefficients in $\Qu$ (or, equivalently, in $\Ze$). Note that, as a result, it is a countable set.
A finite set of real numbers is said to be {\em rationally $($resp., algebra\-ically$)$ linearly 
independent} or simply, {\em rationally $($resp., algebra\-ically$)$ independent}, if
it is linearly independent over the field of rational (resp., algebraic) real numbers.
Furthermore a sequence $(T_i)_{i\ge1}$ of real numbers 
is said to be {\em rationally $($resp., algebraically$)$ linearly 
independent}, if any of its finite subsets
is rationally (resp., algebraically) independent.

In order to define quasiperiodic sets at infinity we will use a definition of quasiperiodic functions adapted in \cite[Definitions 3.1.9.]{fzf} --- the case of finitely many periods and \cite[Definitions 4.6.6]{fzf} --- the case of countably many quasiperiods.\footnote{There exist 
different notions of quasiperiodic and almost periodic functions (and sets) in the literature
on dynamical systems, mathematical physics and harmonic analysis; e.g., \cite{sen},
 \cite{bohr}, \cite{kat}, \cite{lapidusfrank12}, 
\cite[Appendix F]{lapz}, along with the relevant references therein.
The definition we use --- \cite[Definition 3.1.9.]{fzf} is adapted from the one in \cite{enc}.}
Roughly speaking, if the quasiperiods are rationally independent we say that the function is algebraically quasiperiodic, and if the quasiperiods are algebraically independent we say that the function is transcedentally quasiperiodic, while the number of quasiperiods is referred to as the order of the quasiperiodic function, which is also allowed to be countably infinite.

%
%
%
%

In general the set $\mathcal{F}_{qp}$ of quasiperiodic functions is equal to the disjoint union of the set $\mathcal{F}_{tqp}$ of transcendentally quasiperiodic functions
and the set $\mathcal{F}_{aqp}$ of algebraically quasiperiodic functions:
$$
\mathcal{F}_{qp}=\mathcal{F}_{tqp}\cup \mathcal{F}_{aqp}.
$$

\begin{example}\label{ex:qua}
Let $\lambda_i,\nu_i\in\eR$ for $i=1,2$.
If $G(\tau)=\lambda_1G_1(\tau+\nu_1)+\lambda_2G_2(\tau+\nu_2)$, where the functions $G_i$ are nonconstant and $T_i$-periodic for $i=1,2$, such that $T_1/T_2$ is an irrational algebraic number, then $G$
is algebraically $2$-quasiperiodic.
If $T_1/T_2$ is transcendental, then $G$ is transcendentally $2$-quasiperiodic.
\end{example}

We now define quasiperiodic sets at infinity which complements the analogous definition of quasiperiodic bounded sets from~\cite{fzf}.

\begin{definition}[Quasiperiodic set at infinity]\label{quasiperiodic}
Assume $\O\subseteq\eR^N$ is of finite Lebesgue measure and such that it has the following tube formula at infinity:
\begin{equation}\label{Aasymp}
|{_t\O}|=t^{N+D}(G(\log t)+o(1))\q\textrm{as}\q t\to +\ty,
\end{equation}
such that $G$ is nonnegative, $0<\liminf_{\tau\to+\ty} G(\tau)\le\limsup_{\tau\to+\ty} G(\tau)<+\ty$ and $D\in(-\ty,-N]$ is a given constant.\footnote{Note that
it then follows that $\dim_B(\ty,\O)$
exists and is equal to $D$. Moreover, $\unb{\M}^D(\ty,\O)=\liminf_{\tau\to+\ty} G(\tau)$ and $\ovb{\M}^{D}(\ty,\O)=\limsup_{\tau\to+\ty} G(\tau)$.}

We say that $\O$ is an {\em $n$-quasiperiodic set}\label{quasiperiodicdef}
(of {\em order of quasiperiodicity equal to $n$}) {\em at infinity} if the corresponding function $G=G(\tau)$ is $n$-quasiperiodic.

In addition, the set $\O$ is said to be

\smallskip

($a$) {\em transcendentally $n$-quasiperiodic at infinity} if the corresponding function $G$ is transcendentally $n$-quasiperiodic;
\smallskip

($b$) {\em algebraically $n$-quasiperiodic at infinity} if the corresponding function $G$ is algebraically $n$-quasiperiodic.

Furthermore, we also say that $\O$ is {\em $\infty$-quasiperiodic set at infinity} when the function $G$ is $\infty$-quasiperiodic.
Moreover, we make the same distinction between the algebraical and the transcedental case for such sets $\O$.
\end{definition}

In light of Definition~\ref{quasiperiodic} and the comment just before Example \ref{ex:qua},
 one can see that each $n$-quasiperiodic set at infinity is either transcendentally $n$-quasiperiodic at infinity or algebraically $n$-quasiperiodic at infinity.
In other words, the family $\mathscr{D}_{qp}^{\ty}(n)$ of $n$-quasiperiodic sets at infinity is equal to the disjoint union of the family $\mathscr{D}_{tqp}^{\ty}(n)$ of transcendentally $n$-quasiperiodic sets at infinity
and the family $\mathscr{D}_{aqp}^{\ty}(n)$ of algebraically $n$-quasiperiodic sets at infinity:
$$
\mathscr{D}_{qp}^{\ty}(n)=\mathscr{D}_{tqp}^{\ty}(n)\cup \mathscr{D}_{aqp}^{\ty}(n).
$$
Note that the family $(\mathscr{D}_{qp}^{\ty}(n))_{n\ge2}$ is {\em disjoint}, as well as the family $(\mathscr{D}_{tqp}^{\ty}(n))_{n\ge2}$ and the family $(\mathscr{D}_{aqp}^{\ty}(n))_{n\ge2}$. Denoting 
\begin{equation}
\mathscr{D}_{qp}^{\ty}:=\bigcup_{n\ge2}\mathscr{D}_{qp}^{\ty}(n),\q
\mathscr{D}_{tqp}^{\ty}:=\bigcup_{n\ge2}\mathscr{D}_{tqp}^{\ty}(n),\q
\mathscr{D}_{aqp}^{\ty}:=\bigcup_{n\ge2}\mathscr{D}_{aqp}^{\ty}(n),
\end{equation}
we have
$$
\mathscr{D}_{qp}^{\ty}=\mathscr{D}_{tqp}^{\ty}\cup \mathscr{D}_{aqp}^{\ty}.
$$
Theorem~\ref{qp_construction}, or, more precisely, the construction in its proof will show that the families $\mathscr{D}_{tqp}^{\ty}(2)$ and $\mathscr{D}_{aqp}^{\ty}(2)$ are infinite.

\begin{theorem}\label{2-quasi}
The families $\mathscr{D}_{tqp}^{\ty}(2)$ and $\mathscr{D}_{aqp}^{\ty}(2)$ are infinite.
\end{theorem}

\begin{proof}
We note that in the construction of the set $\O$ in the proof of Theorem~\ref{qp_construction} if we only take two sets $\O_{\ty}^{(a_1,b_1)}$ and $\O_{\ty}^{(a_2,b_2)}$ instead of infinitely many, we can construct an algebraically or a transcendentally $2$-quasiperiodic unbounded set at infinity with prescribed box dimension at infinity equal to $D<-2$.
We point out here that the set $\O$ constructed from sets $\O_{\ty}^{(a_1,b_1)}$ and $\O_{\ty}^{(a_2,b_2)}$ has the following tube formula at infinity
\begin{equation}
|K_t\cap\O|=t^{2+D}(G(\log t)+O(t^{-\log_{1/a_1}2}))\quad\textrm{as}\quad t\to +\ty,
\end{equation}
where
\begin{equation}
G(\tau)=2^DG_1(\tau+\log 2)+2^{2D}G_2(\tau+2\log 2)
\end{equation}
is a $2$-quasiperiodic function with
\begin{equation}
G_i(\tau)=\frac{2^{-\left\{\frac{\tau}{\log(1/a_i)}\right\}}}{b_i-1}\left(1+\frac{(a_i^{1-b_i})^{\left\{\frac{\tau}{\log(1/a_i)}\right\}}}{a_i^{1-b_i}-2}\right)
\end{equation}
for $i=1,2$.
As we can see the set $\O$ is then $2$-quasiperiodic at infinity but in the sense of the `cube' tube function at infinity $t\mapsto|K_t(0)^c\cap\O|$.
To get a `proper' $2$-quasiperiodic set at infinity one should mimic this construction in a radial way, i.e., use an analog of sets $\O_{\ty}^{(a_i,b_i)}$ that are ``arranged'' around radial rays emanating from the origin.
We will not get into the details of this construction, but on the other hand, we can use Lemma~\ref{tube_functions} from the Appendix to deduce that if we choose $D\in(-3,-2)$ we do get ``proper'' $2$-quasiperiodic sets at infinity even in the present construction.
More precisely, since $\O$ is contained in a strip of finite width, by Lemma~\ref{tube_functions}, we have that
\begin{equation}\nonumber
\begin{aligned}
|{_t\O}|&=|K_t(0)^c\cap\O|+O(t^{-1})\\
&=t^{2+D}(G(\log t)+O(t^{-\log_{1/a_1}2})+O(t^{-2-D-1}))\\
&=t^{2+D}(G(\log t)+o(1))
\end{aligned}
\end{equation}
as $t\to +\ty$; that is, $\O$ is $2$-quasiperiodic at infinity.

Now, for the algebraical case it suffices to choose $a_1\in(0,1/2)$ and define, for instance, $a_2:=a_{1}^{\sqrt{m}}$ where $m\geq 2$ is an integer that is not a perfect square.
Then we have that $b_1=\log_{1/a_1}2-D-1$ and $b_2=\log_{1/a_2}2-D-1$. Furthermore for the periods we have that
$$
T_1=\log(1/a_1)\quad\textrm{ and }\quad T_2=\log(1/a_2)=\sqrt{m}\log(1/a_1),
$$
i.e., $T_2/T_1=\sqrt{m}$.

On the other hand, if we choose, for instance, $a_1=1/3$ and $a_2=1/k$ where $k>3$ is an integer that is not a power of $3$, we have that
$$
\frac{T_1}{T_2}=\frac{\log{3}}{\log{k}}=\log_k3
$$
which is a transcendental number, a fact that follows from the Gel'fond--Schneider Theorem \cite{gelfond}.
\end{proof}

\begin{remark}\label{c_dim_o}
As a consequence of~\eqref{twop_set}, we have that the complex dimensions at infinity of the set $\O$ from Theorem~\ref{2-quasi} visible through $W:=\{\re s>D-2\}$ are given by
\begin{equation}
\{D-\log_{1/a_i}2:i=1,2\}\cup\left(D+\mathbf{p}(a_1)\I\Ze\right)\cup\left(D+\mathbf{p}(a_2)\I\Ze\right)
\end{equation}
where $\mathbf{p}(a_i)=2\pi/\log(1/a_i)$ for $i=1,2$ are the {\em oscillatory quasiperiods} of $\O$. 
\end{remark}

\begin{theorem}[Construction of $n$-quasiperiodic sets at infinity]\label{n-quasi}
The families $\mathscr{D}_{tqp}^{\ty}(n)$ and $\mathscr{D}_{aqp}^{\ty}(n)$ are infinite for every integer $n\geq 2$.
\end{theorem}

\begin{proof}
The proof is analogous to the proof of Theorem~\ref{2-quasi} the difference being in the fact that we take $n$ sets $\O_{\ty}^{(a_i,b_i)}$, for $i=1,\ldots,n$ instead of only two.
In that way we construct a set $\O$ with $n$ quasiperiods at infinity which will be `proper' $n$-quasiperiodic if we additionally restrict ourselves to $D\in(-3,-2)$.
(See the discussion in the proof of Theorem~\ref{2-quasi}.)
For the algebraically $n$-quasiperiodic case we may choose $a_1\in(0,1/2)$ and define $a_{i+1}:=a_1^{\sqrt{p_{i}}}$ where $p_i$ is the $i$-th prime number for $i\geq 1$.
Then for the quasiperiods of $\O$ we have that
$$
T_1=\log(1/a_1)\quad\textrm{and}\quad T_{i+1}=\log(1/a_1^{\sqrt{p_{i}}})=T_1\sqrt{p_{i}}
$$
for $i\geq 1$.
It is obvious that the quasiperiods $T_1,\ldots,T_n$ are algebraically dependent.
On the other hand, they are rationally independent.
Namely suppose that there are $\lambda_1,\ldots,\lambda_n\in\Qu$ such that
$$
\lambda_1T_1+\lambda_2T_2+\ldots+\lambda_nT_n=0.
$$
This is equivalent to
$$
\lambda_1+\lambda_2\sqrt{2}+\cdots+\lambda_n\sqrt{p_{n-1}}=0
$$
which is possible only if $\lambda_1=\cdots=\lambda_n=0$ according to a result of Besicovitch \cite{Bes}.
This proves that the set $\O$ indeed is algebraically $n$-quasiperiodic at infinity.

Let us now construct a transcendentally $n$-quasiperiodic set at infinity.
We choose now $a_i:=1/p_{i+1}$ with $p_i$ being the $i$-th prime number for $i\geq 1$.
Note that now $T_i=\log{(1/a_i)}=\log p_{i+1}$ and these numbers are rationally independent.
Indeed, if we assume the contrary; that is, that there exist rational numbers $\lambda_1,\ldots,\lambda_n$ such that $\sum_{i=1}^{n}\lambda_i\log p_{i+1}=0$, then this implies that $\prod_{i=1}^{n}p_{i+1}^{\lambda_i}=1$ which is in contradiction with the Fundamental theorem of algebra.
Next, Baker's Theorem \cite[Theorem~2.1]{baker} (a nontrivial extension of the Gel'fond--Schneider Theorem) implies that the numbers $T_1,\ldots,T_n$ are also algebraically independent; that is, the set $\O$ is transcendentally $n$-quasiperiodic.
\end{proof}

\begin{remark}
Similarly as in Remark~\ref{c_dim_o}, the set $\O$ constructed in Theorem~\ref{n-quasi} will have the following set of complex dimensions visible through $W=\{\re s>D-2\}$:
\begin{equation}
\bigcup_{i=1}^{n}\left(\{D-\log_{1/a_i}2\}\cup\left(D+\mathbf{p}(a_i)\I\Ze\right)\right)
\end{equation}
where $\mathbf{p}(a_i)=2\pi/\log(1/a_i)$ for $i=1,\ldots,n$ are the oscillatory quasiperiods of $\O$ at infinity.
\end{remark}

\begin{remark}
It is clear that one can construct somewhat more general examples of $n$-quasiperiodic sets at infinity than the ones from the proof of Theorem~\ref{n-quasi} by choosing other admissible values for the parameters $a_i$. 
\end{remark}

Let us conclude this section by defining the notion of $\ty$-quasiperiodic sets at infinity and showing that the maximally hyperfractal set $\O$ at infinity from Theorem~\ref{qp_construction} gives an example of such a set.
Moreover, by carefully choosing the parameters $a_i$ we can construct an infinite number of algebraically and transcendentally $\ty$-quasiperiodic sets at infinity.

In much the same way as before, if we denote with $\mathscr{D}_{qp}^{\ty}(\ty)$ the family of all $\ty$-quasiperiodic sets at infinity, then it is clear that this family is a disjoint union of $\mathscr{D}_{aqp}^{\ty}(\ty)$ and $\mathscr{D}_{tqp}^{\ty}(\ty)$; that is, the algebraically $\ty$-quasiperiodic subfamily and the transcendentally $\ty$-quasiperiodic subfamily, respectively.

\begin{theorem}[Construction of $\infty$-quasiperiodic sets at infinity]\label{ty_quasi}
The families $\mathscr{D}_{aqp}^{\ty}(\ty)$ and $\mathscr{D}_{tqp}^{\ty}(\ty)$ are infinite.
\end{theorem}

\begin{proof}
For a fixed $D<-2$ a member of each subfamily is the maximal hyperfractal $\O$ at infinity constructed in Theorem~\ref{qp_construction} for a specifically chosen  sequence of parameters $a_i$.
More precisely, to get a `proper' $\ty$-quasiperiodic set at infinity, we have to choose $D\in(-3,-2)$.
(See the discussion in the proof of Theorem~\ref{2-quasi}.)
We proceed analogously as in the proof of Theorem~\ref{n-quasi}; that is, let $(p_i)_{i\geq 1}$ be the increasing sequence of all prime numbers.
For the algebraically $\ty$-quasiperiodic set at infinity we may choose $a_1\in(0,1/2)$ and define $a_{i+1}:=a_1^{\sqrt{p_i}}$ for $i\geq 1$.
Similarly as before, from \cite{Bes} we easily conclude that the sequence of quasiperiods $T_i=\log(1/a_i)$, $i\geq 1$ is rationally independent.

On the other hand, for the transcendentally $\ty$-quasiperiodic set at infinity we may choose $a_i:=1/p_{i+1}$ for $i\geq 1$ and, similarly as before, \cite[Theorem~2.1]{baker} assures that the the sequence of quasiperiods $T_i=\log(1/a_i)$, $i\geq 1$ is algebraically independent.
\end{proof}

\section{Appendix}

Here we state and prove a result that is needed when constructing quasiperiodic sets at infinity in Section~\ref{qp_sets}.
Recall that the volume of the $N$-dimensional unit ball  is given by
$
\omega_N:=\frac{\pi^{N/2}}{\Gamma\left(\frac{N}{2}+1\right)}
$
where $\Gamma$ denotes the Gamma function.

\begin{lemma}\label{tube_functions}
Let $\O\subseteq\eR^N$ with $|\O|<\ty$ be such that it is contained in a cylinder
\begin{equation}
x_2^2+x_3^2+\cdots+x_N^2\leq C
\end{equation}
for some constant $C>0$ where $x=(x_1,\ldots,x_N)$.
Then
\begin{equation}
|_t\O|=|K_t(0)^c\cap\O|+O(t^{-1})\quad\mathrm{as}\quad t\to +\ty,
\end{equation}
where $K_t(0)$ is the ball of radius $t$ centered at $0$ in the max norm, $|\cdot|_{\infty}$ on $\mathbb{R}^N$.
\end{lemma}

\begin{proof}
We note that for $t$ sufficiently large the difference $|_t\O|-|K_t(0)^c\cap\O|$ is less than the volume of the $N$-dimensional cylinder of height $h:=t-\sqrt{t^2-C^2}$ with base of radius $C$.
In other words, we have that
$$
\big||_t\O|-|K_t(0)^c\cap\O|\big|\leq h\omega_{N-1}C^{N-1}=\frac{\omega_{N-1}C^{N+1}}{t+\sqrt{t^2+C^2}}=O(t^{-1}).
$$
\end{proof} 

%

We also state here a simple and useful corollary of the Complex mean value Theorem~{\cite[Theorem~2.2]{EvJa}} that is needed in the proof of Theorem \ref{equiv_mero}.

\begin{corollary}\label{comp_ocj}
Let $f$ be a holomorphic function defined on an open convex subset $U_f$ of $\Ce$.
Furthermore let $a$ and $b$ be two distinct points in $U_f$.
\begin{equation}
|f(b)-f(a)|\leq\sqrt{2}|b-a|\max_{s\in[a,b]}|f'(s)|.
\end{equation}
\end{corollary}

\begin{proof}
From \cite[Theorem~2.2]{EvJa}, we have that there are $s_1,s_2\in(a,b)$ such that\footnote{Here, we also use the obvious inequalities $|\re s|,|\im s|\leq|s|$ for $s\in\Ce$.}
$$
\begin{aligned}
\left|\frac{f(b)-f(a)}{b-a}\right|^2&=|\re(f'(s_1))|^2+|\im(f'(s_2))|^2\\
&\leq|f'(s_1)|^2+|f'(s_2)|^2\\
&\leq 2\max_{s\in[a,b]}|f'(s)|^2.
\end{aligned}
$$
Taking the square root of both sides and multiplying by $|b-a|$ completes the proof of the corollary.
\end{proof}

\end{document}